\documentclass[pdflatex,sn-basic,Numbered]{sn-jnl}

\usepackage{graphicx}%
\usepackage{multirow}%
\usepackage{amsmath,amssymb,amsfonts}%
\usepackage{amsthm}%
\usepackage{mathrsfs}%
\usepackage[title]{appendix}%
\usepackage{xcolor}%
\usepackage{textcomp}%
\usepackage{manyfoot}%
\usepackage{booktabs}%
\usepackage{algpseudocode}%
\usepackage{listings}%
\usepackage{cleveref}
\usepackage{verbatim}

\newcommand{\rd}{\, \mathrm{d}}
\newcommand{\rand}{\mathrm{rand}}

\newcommand{\bszero}{\boldsymbol{0}}
\newcommand{\bsh}{\boldsymbol{h}}

\newcommand{\bsl}{\boldsymbol{\ell}}
\newcommand{\bsx}{\boldsymbol{x}}
\newcommand{\bsy}{\boldsymbol{y}}
\newcommand{\bsz}{\boldsymbol{z}}
\newcommand{\bsgamma}{\boldsymbol{\gamma}}
\newcommand{\bsDelta}{\boldsymbol{\Delta}}
\newcommand{\EE}{\mathbb{E}}

\newcommand{\NN}{\mathbb{N}}

\newcommand{\RR}{\mathbb{R}}
\newcommand{\ZZ}{\mathbb{Z}}
\newcommand{\Acal}{\mathcal{A}}
\newcommand{\Hcal}{\mathcal{H}}
\newcommand{\Pcal}{\mathcal{P}}

\newcommand{\tmod}[1]{{\;(\text{mod}\; #1)}}

\newcommand{\e}{\mathbb{E}}
\newcommand{\bsone}{\boldsymbol{1}}

\newcommand{\icomp}{\mathtt{i}}
\newcommand{\abs}[1]{\left\vert#1\right\vert}
\newcommand{\norm}[1]{\left\Vert#1\right\Vert}

\allowdisplaybreaks

\theoremstyle{thmstyleone}%
\newtheorem{theorem}{Theorem}
\newtheorem{algorithm}[theorem]{Algorithm}
\newtheorem{remark}[theorem]{Remark}
\newtheorem{lemma}[theorem]{Lemma}

\newtheorem{corollary}[theorem]{Corollary}

\theoremstyle{thmstyletwo}%
\newtheorem{case}{Case}

\begin{document}

\title[$L_2$-approximation using median lattice algorithms]{{$L_2$}-approximation using median lattice algorithms}

\author[1]{\fnm{Zexin} \sur{Pan}}\email{zexin.pan@oeaw.ac.at}

\author[1]{\fnm{Peter} \sur{Kritzer}}\email{peter.kritzer@oeaw.ac.at}

\author*[2]{\fnm{Takashi} \sur{Goda}}\email{goda@frcer.t.u-tokyo.ac.jp}

\affil[1]{\orgdiv{Johann Radon Institute for Computational and Applied Mathematics (RICAM)}, \orgname{Austrian Academy of Sciences}, \orgaddress{\street{Altenbergerstr.\ 69}, \city{Linz}, \postcode{4040}, \country{Austria}}}

\affil*[2]{\orgdiv{Graduate School of Engineering}, \orgname{The University of Tokyo}, \orgaddress{\street{7-3-1 Hongo, Bunkyo-ku}, \city{Tokyo}, \postcode{113-8656}, \country{Japan}}}

\abstract{In this paper, we study the problem of multivariate $L_2$-approximation of functions belonging to a weighted Korobov space. We propose and analyze a median lattice-based algorithm, inspired by median integration rules, which have attracted significant attention in the theory of quasi-Monte Carlo methods. Our algorithm approximates the Fourier coefficients associated with a suitably chosen frequency index set, where each coefficient is estimated by taking the median over approximations from randomly shifted rank-1 lattice rules with independently chosen generating vectors. We prove that the algorithm achieves, with high probability, a convergence rate of the $L_2$-approximation error that is arbitrarily close to optimal with respect to the number of function evaluations. Furthermore, we show that the error bound depends only polynomially on the dimension, or is even independent of the dimension, under certain summability conditions on the weights. Numerical experiments illustrate the performance of the proposed median lattice-based algorithm.}



\maketitle

\section{Introduction}

In this paper, we study the problem of multivariate $L_2$-approximation in \textit{weighted Korobov spaces} of smoothness $\alpha$, denoted by $\Hcal_{d,\alpha,\bsgamma}$. Our goal is to derive a probabilistic result that yields a competitive error bound with high probability. This result, in turn, leads to an almost optimal rate of the randomized (worst-case root-mean-squared) $L_2$-approximation error. In this context, $d$ is the number of variables the elements of the space depend on (frequently referred to as the \textit{dimension} of the approximation problem), $\alpha$ is the \textit{smoothness parameter} and $\bsgamma=(\gamma_{\mathfrak{u}})_{\mathfrak{u}\in \{1{:}d\}}$ is a collection of \textit{weights}. Here and in the following, we write $\{1{:}V\}$ to denote the set $\{1,\ldots,V\}$ for a positive integer $V$. The basic idea of weights is to assign a positive number 
to every possible group of variables $\mathfrak{u}\in \{1{:}d\}$, which describes how much influence the respective groups of variables have in the problem. Large values of $\gamma_{\mathfrak{u}}$ indicate higher influence, low values of $\gamma_{\mathfrak{u}}$ indicate less influence. 
This concept was introduced in the seminal paper \cite{sloan1998when}.
We will give the detailed definition of the spaces $\Hcal_{d,\alpha,\bsgamma}$ and further comments on the weights in \Cref{sec:preliminaries}.

Korobov spaces are Hilbert spaces of $d$-variate, one-periodic functions with absolutely convergent Fourier series. These spaces have been studied frequently in the context of numerical integration and function approximation (see, e.g., \cite{dick2022lattice, kuo2013high, sloan1994lattice} for overviews), but have---in a slightly different form---recently also attracted interest in the literature on neural networks, see \cite{liu2024approx}.

A natural approach for approximating functions in Korobov spaces is to truncate the Fourier series first and then approximate the remaining Fourier coefficients, see, e.g., \cite{kuo2006lattice, novak2004tractability} and the references therein. Specifically, we define a finite set $\Acal\subset\ZZ^d$, and approximate $f\in \Hcal_{d,\alpha,\bsgamma}$ by truncating its Fourier expansion to the frequencies in $\Acal$ first:
\[ f(\bsx)=\sum_{\bsh\in \ZZ^d}\widehat{f}(\bsh)\exp(2\pi \icomp\bsh\cdot \bsx) \approx \sum_{\bsh\in \Acal}\widehat{f}(\bsh)\exp(2\pi \icomp\bsh\cdot \bsx),\]
where 
\[
 \widehat{f}({\bsh}):=\int_{[0,1]^d} f(\bsx) \exp({-2\pi\icomp \bsh\cdot \bsx}) \rd \bsx
\]
denotes the $\bsh$-th Fourier coefficient of $f$ and $\cdot$ denotes the standard inner product in $\RR^d$.
Here, the set $\Acal$ needs to be chosen carefully so that we can handle the truncation error and, simultaneously, that the computational cost of an approximation algorithm based on $\Acal$ is not too high. 

Once we have selected $\Acal$, the next step is to approximate each Fourier coefficient $\widehat{f}(\bsh)$ for every $\bsh\in\Acal$. Since each Fourier coefficient is defined as an integral over $[0,1]^d$, as shown above, a suitable integration rule can be used for this approximation. In this paper, we employ \textit{quasi-Monte Carlo (QMC) rules} for this purpose. QMC rules are equal-weight quadrature rules of the form 
\[
Q_{N,d}(g)=\frac{1}{N}\sum_{k=0}^{N-1} g(\bsx_k),
\]
where the integration nodes $\bsx_0,\ldots,\bsx_{N-1}$ are chosen deterministically in a way that matches the properties of the integrand $g$. The standard choice for integrating functions in Korobov spaces is to use \textit{lattice point sets} as integration nodes, yielding \textit{lattice rules} $Q_{N,d}$ (we again refer to \Cref{sec:preliminaries} for details). This principal idea goes back to Korobov, see \cite{korobov1963numbertheoretic} and also \cite{hua1963applications}. We will mostly follow the paper \cite{kuo2006lattice} in our notation, but we also refer to \cite{li2003trigonometric,temlyakov1985approximaterussian, temlyakov1993approximate, zeng2006error} for related results.

As shown in \cite{kuo2006lattice} and related works, using a single lattice rule with $M$ nodes to approximate the Fourier coefficients yields a convergence rate of the worst-case (i.e., considering the supremum over all functions in the unit ball of the space) $L_2$-approximation error (see \Cref{sec:preliminaries}) arbitrarily close to $\mathcal{O}(M^{-\alpha/2})$. In a slightly more general setting, \cite{byrenheid2017tight} proves that the rate $\mathcal{O}(M^{-\alpha/2})$ cannot be improved with such an approach for any $d\ge 2$. However, as shown in \cite{kaemmerer2019constructing} and \cite{kammerer2019approximation}, this convergence rate can be improved in a probabilistic sense by employing so-called \textit{multiple lattices}. Moreover, in the recent paper \cite{CGK24} by Cai, Goda, and Kazashi, the authors introduced a randomized lattice-based algorithm for $L_2$-approximation and proved that one can achieve a convergence rate arbitrarily close to $\mathcal{O}(M^{-\alpha/2 - 1/8})$ of the worst-case root-mean-squared $L_2$-approximation error. The same paper proves that this randomized algorithm cannot achieve a better rate than $\mathcal{O}(M^{-\alpha/2 - 1/2})$.

The key idea that we present in this paper is to adopt a probabilistic framework by using so-called \textit{median lattice rules} to approximate all Fourier coefficients $\widehat{f}(\bsh)$, $\bsh\in \Acal$, instead of relying on a single fixed lattice rule. Median integration rules, as previously studied in, e.g., \cite{goda2022free, goda2024universal, pan2023super-pol, pan2024super-pol}, rely on a randomly chosen or suitably randomized point set within a QMC rule. By repeating this random experiment $R$ times and then taking the median of the resulting estimates, one obtains a quadrature rule that can often be applied over a wide range of parameters, such as smoothness and weights, without the need to specify them in advance, while still achieving excellent convergence rates.
In the present context, we apply a median rule based on $R$ independent lattice rules, each using $N$ nodes, requiring a total of $M=RN$ function evaluations.
The number of repetitions $R$ must grow with $N$, but only at a rate proportional to $\log N$, making the approach computationally feasible even for large $N$.

Our main result shows that the proposed algorithm achieves a convergence rate arbitrarily close to $\mathcal{O}(M^{-\alpha})$ with high probability for individual functions in $\Hcal_{d,\alpha,\bsgamma}$.
Moreover, we show that, under suitable summability conditions on the weights $\bsgamma$, the dependence of both the required number of repetitions $R$ and the probabilistic error bound on the dimension $d$ can be significantly reduced or even eliminated. This property links our result to the field of \emph{Information-Based Complexity}, see \cite{novak2008tractability}, and ensures that high-dimensional problems remain computationally feasible.
Although this main result does not necessarily imply that the proposed algorithm achieves the \emph{worst-case} $L_2$-approximation error of $\mathcal{O}(M^{-\alpha+\varepsilon})$ with high probability (see \Cref{sec:discussion}), it can be shown that the algorithm does achieve the \emph{randomized} (worst-case root-mean-squared) $L_2$-approximation error of $\mathcal{O}(M^{-\alpha+\varepsilon})$.
Importantly, this convergence rate is optimal in the sense that no randomized algorithm based on $M$ function evaluations can achieve a better rate (see again \Cref{sec:discussion}).

The rest of the paper is structured as follows. 
In \Cref{sec:preliminaries}, we define the space $\Hcal_{d,\alpha,\bsgamma}$, formally state the approximation problem, and introduce rank-1 lattice rules. 
\Cref{sec:median} presents our median lattice-based algorithm and shows the main result of this paper on the $L_2$-approximation error bound. 
\Cref{sec:tractability} examines to what extent and under what conditions the curse of dimensionality can be avoided in our approximation error bounds.
\Cref{sec:numerics} reports numerical results illustrating the performance of the proposed algorithm.
The paper concludes with a discussion in \Cref{sec:discussion}.

\section{Preliminaries and notation}\label{sec:preliminaries}

Below, we denote the set of integers by $\ZZ$ and the set of positive integers by $\NN$. Furthermore, as usual, we denote the exponential function by $\exp$, and write $\exp(1)=e$ for short. Bold symbols denote vectors, where the length is usually clear from the context. Further notation will be introduced as necessary.

\subsection{Weighted Korobov spaces}
Let $\bsgamma=(\gamma_{\mathfrak{u}})_{\mathfrak{u}\in\{1{:}d\}}$ be a collection of positive reals, which we refer to as the weights in the following. As mentioned in the introduction, the weight $\gamma_{\mathfrak{u}}$ models the influence of the group of variables $x_j$ with indices $j\in \mathfrak{u}$. In this paper, we restrict ourselves to the most common case of weights, namely \textit{product weights}, 
which are characterized by a non-increasing positive sequence $(\gamma_j)_{j\ge 1}$, where we assume that $1\ge \gamma_1 \ge \gamma_2 \ge \cdots$. In this case we choose 
\[
\gamma_{\mathfrak{u}}:=\prod_{j\in \mathfrak{u}} \gamma_j,
\]
and put $\gamma_{\emptyset}:=1$.
We remark that also other choices of weights are considered in the literature, for which one may also allow single weights $\gamma_{\mathfrak{u}}$ to be zero. However, to avoid technical difficulties, we restrict ourselves to the product weight case in the present paper. We also note that the assumption that $1\ge \gamma_1$ could be dropped without changing the nature of the results in this paper, but resulting in additional technical notation. For the sake of simplicity, we therefore stay with the assumption $1\ge \gamma_1$ in the following. For other kinds of weights and the corresponding function spaces, see, e.g., \cite{dick2022lattice, kuo2013high}.

For product weights $\bsgamma$ as above, and a real number $\alpha>1/2$, we define the function
\[
r_{2\alpha,\bsgamma} (\bsh):=\prod_{j=1}^d \max \left\{\frac{\abs{h_j}^{2\alpha}}{\gamma_j}, 1\right\}
\]
for $\bsh=(h_1,\ldots,h_d)\in \ZZ^d$. 

We now define the weighted Korobov space $\Hcal_{d,\alpha,\bsgamma}$, which is a reproducing kernel Hilbert space consisting of one-periodic functions on $[0,1]^d$, where periodicity is understood with respect to each variable. The reproducing kernel and the inner product in $\Hcal_{d,\alpha,\bsgamma}$ are given, respectively, by
\[ K_{d, \alpha, \bsgamma}(\bsx, \bsy)=\sum_{\bsh \in \ZZ^{d}} \frac{\exp (2 \pi i \bsh \cdot(\bsx-\bsy))}{r_{2\alpha,\bsgamma}(\bsh)}, \]
and
\[ \langle f, g\rangle_{d, \alpha, \bsgamma}=\sum_{\bsh \in \ZZ^{d}}\hat{f}(\bsh)\, \overline{\hat{g}(\bsh)}\, r_{2\alpha,\bsgamma}(\bsh), \] 
where, as already defined in the introduction, $\widehat{f}({\bsh})$ denotes the $\bsh$-th Fourier coefficient of $f$.
The induced norm is then given by $\|f\|_{d,\alpha,\bsgamma}=\sqrt{\langle f, f\rangle_{d, \alpha, \bsgamma}}$.

Note that $\Hcal_{d,\alpha,\bsgamma}$ is a subspace of $L_2 ([0,1]^d)$, and that every element of the Korobov space with $\alpha>1/2$ can be represented pointwise by its Fourier series. Moreover, it is known that if $\alpha\in\NN$, we have that $\Hcal_{d,\alpha,\bsgamma}$ consists of all functions whose mixed partial derivatives up to order $\alpha-1$ in each variable are absolutely continuous and whose mixed derivatives of order $\alpha$ belong to $L_2 ([0,1]^d)$ (see \cite{dick2022component,kritzer2019lattice} for details). For this reason, $\alpha$ is referred to as the smoothness parameter. 

\subsection{$L_2$-approximation}\label{subsec:approximation}
Our goal is to study $L_2$-approximation of functions $f\in \Hcal_{d,\alpha,\bsgamma}$. Formally, this corresponds to approximating the embedding operator $S:\Hcal_{d,\alpha,\bsgamma}\to L_2 ([0,1]^d)$, $S(f)=f$. For a given deterministic approximation algorithm $A:\Hcal_{d,\alpha,\bsgamma}\to L_2 ([0,1]^d)$, a common error measure is the worst-case error, defined as
\begin{align}\label{eq:worst-case}
  \mathrm{err}(\Hcal_{d,\alpha,\bsgamma}, L_2, A):= \sup_{\substack{f\in \Hcal_{d,\alpha,\bsgamma}\\ 
  \norm{f}_{d,\alpha,\bsgamma}\le 1}}
  \|f-A(f)\|_{L_2},
\end{align}
i.e., one considers the worst performance of $A$ over the unit ball of $\Hcal_{d,\alpha,\bsgamma}$. 

When we speak of a randomized approximation algorithm, we refer to a pair consisting of a probability space $(\Omega, \Sigma, \mu)$ and a family of mappings $A = (A^\omega)_{\omega \in \Omega}$, where each $A^\omega$ is a deterministic approximation algorithm for fixed $\omega \in \Omega$. 
In the randomized setting, a common error measure is the worst-case root-mean-squared error, also referred to as the \emph{randomized error}, defined by
\begin{align}\label{eq:worst-case_randomized}
  \mathrm{err}^{\mathrm{ran}}(\Hcal_{d,\alpha,\bsgamma}, L_2, (A^{\omega})):= \sup_{\substack{f\in \Hcal_{d,\alpha,\bsgamma}\\ 
  \norm{f}_{d,\alpha,\bsgamma}\le 1}}
  \left(\EE_{\omega}\left[\|f-A^{\omega}(f)\|^2_{L_2}\right]\right)^{1/2}.
\end{align}
Alternatively, one can consider the $(\epsilon,\delta)$-approximation framework, see for instance \cite{kunsch2019optimal}, which asks for which pairs $(\epsilon,\delta)$ the algorithm satisfies
\begin{align}\label{eq:eps-delta_framework} 
\mathrm{Pr}_{\omega}\left[\|f-A^{\omega}(f)\|_{L_2}>\epsilon \right]\le \delta \qquad \text{for all $f$ with $\norm{f}_{d,\alpha,\bsgamma}\le 1$.} \end{align}

Our main result (see \Cref{cor:error_bound} below) states that, with high probability, a median-based lattice algorithm can achieve an error bound of order $M^{-\alpha+\varepsilon}$ for individual functions in $\Hcal_{d,\alpha,\bsgamma}$, where $M$ denotes the total number of function evaluations used in $A$ (see \Cref{sec:median} for the precise definition).
This is equivalent to saying that our proposed algorithm satisfies the $(\epsilon,\delta)$-approximation framework \eqref{eq:eps-delta_framework} with small $\delta$ and $\epsilon = \mathcal{O}(M^{-\alpha+\varepsilon})$ for arbitrarily small $\varepsilon$.
This result, in turn, leads to an almost optimal rate of the randomized $L_2$-approximation error (see \Cref{sec:discussion}).
The algorithm is based on rank-1 lattice rules, introduced in the next subsection, and uses $R$ randomized instances of such rules, each with $N$ nodes. Thus, the total cost is $M = RN$.

\subsection{Rank-1 lattice rules}
Let us now introduce rank-1 lattice rules, which are a key building block of our median algorithm. For simplicity, we assume in the following that $N$ is a prime number. Let $\bsz=(z_1,\ldots,z_d)$ be a 
vector with each component in $\{1{:}(N-1)\}$. 
Using this vector, we define a lattice point set $\Pcal_N=\{\bsx_0,\ldots,\bsx_{N-1}\}$ as follows: 
for each $j\in\{1{:}d\}$ and $k\in\{0,\ldots,N-1\}$, the $j$-th component of the $k$-th point $\bsx_k$ is given by
\[
x_j^{(k)}:=\left\{\frac{k z_j}{N}\right\},
\]
where $\{y\}=y-\lfloor y \rfloor$ denotes the fractional part of a non-negative real number $y$. 
By repeating the above procedure for all $j\in\{1{:}d\}$ and $k\in\{0,\ldots,N-1\}$, we obtain the full lattice point set $\Pcal_N$. A QMC rule using $\Pcal_N$ is called a rank-1 lattice rule. 

Note that, for fixed $N$ and $d$, a lattice rule is fully characterized by the choice of the \textit{generating vector} $\bsz$. However, not all choices of $\bsz$ lead to lattice rules of sufficient quality for integration or approximation purposes. Fortunately, there exist fast construction algorithms that provide good generating vectors for given $N$, $d$, $\alpha$, and $\bsgamma$. For comprehensive overviews of the theory of lattice rules, we refer to the books \cite{dick2022lattice, sloan1994lattice}, and in particular to \cite[Chapters 3 and 4]{dick2022lattice} for details on the construction of good lattice rules. In this paper, however, we employ a randomized approach, in which the generating vectors $\bsz$ are chosen at random to generate the lattice point sets. Hence, we do not need to be concerned with construction algorithms. 

It is sometimes useful, and we shall also adopt this approach here, to introduce an additional random element when applying lattice rules. This is commonly achieved by a \textit{random shift} $\Delta\in [0,1)^d$. Given a lattice point set $\Pcal_N=\{\bsx_0,\ldots,\bsx_{N-1}\}$ and drawing $\bsDelta$ from the uniform distribution over $[0,1)^d$, the corresponding \textit{randomly shifted lattice rule} applied to a function $g$ is given by 
\[
   Q_{N,d,\bsDelta}(g):=\frac{1}{N}\sum_{k=0}^{N-1}g \left(\left\{\bsx_k + \bsDelta\right\}\right),
\]
i.e., all points of $\Pcal_N$ are shifted modulo one by the same $\bsDelta$.
Here, the operation $\{\cdot\}$ is applied component-wise to a vector.
If we wish to emphasize the role of the generating vector $\bsz$ of $\Pcal_N$, we will write $Q_{N,d,\bsz,\bsDelta}$ instead.

\section{The median algorithm and its error}\label{sec:median}
In this section, we introduce our median lattice-based algorithm and analyze its $L_2$-approximation error. 

\subsection{Median lattice-based algorithm}
We assume product weights $\bsgamma$ with $\gamma_j \in (0,1]$ for all $j\ge 1$, and a smoothness parameter $\alpha>1/2$.
These parameters are assumed fixed but otherwise arbitrary. 

For a real $L\ge 1$, we define
\[
\Acal_d(L) := \left\{ \bsh \in \ZZ^d : r_{2\alpha, \bsgamma}(\bsh) \leq L^{2\alpha} \right\}=\left\{ \bsh \in \ZZ^d : r_{1, \bsgamma^{1/(2\alpha)}}(\bsh) \leq L\right\}.
\]
Additionally, for $R\in\NN$, we define the median of the complex numbers $Z_1,\dots,Z_R$ as
\[ 
\mathrm{median}_{r\in\{1{:}R\}}(Z_r)=\mathrm{median}_{r\in\{1{:}R\}} \Re(Z_r) + \icomp\cdot \mathrm{median}_{r\in\{1{:}R\}} \Im (Z_r). 
\]
When $R$ is odd, the median of $R$ real numbers is uniquely defined, ensuring that the above definition is well-defined. Although the case with even $R$ can be analyzed in a similar manner, we assume in the following that $R$ is an odd integer. Moreover, we will assume that $N$ is a prime 
number. While this assumption is not strictly necessary for all results below, it simplifies certain aspects of the analysis and will be adopted henceforth.

We now introduce the \emph{median lattice-based algorithm} for $L_2$-approximation in the weighted Korobov space $\Hcal_{d,\alpha,\bsgamma}$. 
\begin{algorithm}\label{alg:median}
    Let $\alpha>1/2$ and $d \in \NN$, and let $\gamma_j \in (0,1]$ for $j = 1, \ldots, d$. Given an odd number of repetitions $R \in \NN$, a parameter $\tau > 0$, and a prime number $N$, do the following:
    \begin{enumerate}
        \item Define 
        \begin{equation}\label{eqn:PN}
        P_N(\tau,d,\bsgamma):=\prod_{j=1}^d \left(1+ 2\gamma_j^{1/(2\alpha)} (1+\tau\log N)\right),
        \end{equation}
        and
        \begin{equation}\label{eqn:Nstar}
        N_*:=\frac{N-1 }{\exp(\tau^{-1})P_{N}(\tau,d,\bsgamma)}.
        \end{equation}
        \item \textbf{For} $r$ from $1$ to $R$, do the following:
        \begin{enumerate}
            \item Randomly draw $\bsz_{r}$ from the uniform distribution over the set $\{1{:}(N-1)\}^d$.
            \item Randomly draw $\bsDelta_r$ from the uniform distribution over 
            $[0,1)^d$.
            \item For all $\bsh\in \Acal_d(N_*)$, set
            \[ 
            \widehat{f}_{N,\bsz_r,\bsDelta_r}(\bsh):=
            \frac{1}{N} \sum_{k=0}^{N-1} f\left( \left\{ \frac{k \bsz_r}{N}+\bsDelta_r \right\}\right)\exp\left(-2\pi \icomp \bsh \cdot \left( \frac{k \bsz_r}{N}+\bsDelta_r\right)\right).
            \]   
        \end{enumerate}
       \textbf{end for}
       \item For all $\bsh\in \Acal_d(N_*)$, define the aggregation
       \[ 
            \widehat{f}_{N,\bsz_{\{1{:}R\}},\bsDelta_{\{1{:}R\}}}(\bsh):=\mathrm{median}_{r\in\{1{:}R\}} \Big(\widehat{f}_{N,\bsz_r,\bsDelta_r}(\bsh)\Big),
       \]
       where $\bsz_{\{1{:}R\}},\bsDelta_{\{1{:}R\}}$ denote the collections of $\bsz_r$ and $\bsDelta_r$, for $r\in\{1{:}R\}$, respectively.
       Define the final approximation by
        \[ 
        (A^{\rand}_{N,\bsz_{\{1{:}R\}},\bsDelta_{\{1{:}R\}},\Acal_d(N_*)}(f))(\bsx):= \sum_{\bsh \in \Acal_d(N_*)} \widehat{f}_{N,\bsz_{\{1{:}R\}},\bsDelta_{\{1{:}R\}}}(\bsh) \exp (2 \pi \icomp \bsh \cdot \bsx),\]
\end{enumerate}
\end{algorithm}

Before proceeding to the error analysis, we summarize some remarks concerning the algorithm.

\begin{remark}
    Note that for computing $\widehat{f}_{N,\bsz_r,\bsDelta_r}(\bsh)$, the function evaluations required of $f$ do not depend on $\bsh$. Thus, a total of $RN$ evaluations of $f$ are required to run \Cref{alg:median}. Furthermore, due to the independence between different repetitions $r$, the second step of \Cref{alg:median} can be performed in parallel.
\end{remark}

\begin{remark}\label{rem:limitation}
    The use of the \emph{median} aggregation over $R$ independent estimators mitigates the impact of outliers caused by unfavorable choices of generating vectors $\bsz_r$ or shifts $\bsDelta_r$. This approach follows the same principle as the median trick used in recent works on numerical integration \cite{goda2022free,goda2024universal,pan2023super-pol,pan2024super-pol}. A key distinction from \cite{goda2022free} is that our approximation algorithm requires the input of the parameter $\alpha$ and the weights $\gamma_j$, which are necessary for defining the index set $\Acal_{d}(N_*)$. Specifically, the quantity $N_*$ is defined in \eqref{eqn:Nstar}, controlling the trade-off between truncation error and the error arising from the estimation of Fourier coefficients. Making the algorithm universal, i.e., applicable without specifying these parameters, will be explored in a subsequent work.
\end{remark}

\begin{remark} 
    For an implementation of \Cref{alg:median}, \Cref{rmk:choosingRN} specifies the requirements on $N$ and $R$ for the algorithm to succeed with $1-\delta$ probability, given $\tau > 0$ and $\delta \in (0,1)$. \Cref{rmk:choosingtau} provides further guidance on selecting the parameter $\tau$ when there is a budget limit on the total number of function evaluations, i.e., when $RN \leq M_{\max}$.
\end{remark}

\subsection{Error analysis}
To prove our main result, we begin by presenting several lemmas. The first lemma improves upon Lemma 1 from~\cite{kuo2006lattice} concerning the cardinality of the hyperbolic cross $\Acal_d(L)$. The subsequent corollary applies this result to bound the size of the index set $\Acal_d(N_*)$ used in \Cref{alg:median}.

\begin{lemma}\label{lem:Adcount} 
For $L\geq 1$, we have  
\begin{equation}\label{eqn:minoverq}
    |\Acal_d(L)|\leq 1+ \inf_{q> 1} \frac{L^{q}}{\zeta(q)} \prod_{j=1}^d \Big(1+2\gamma^{q/(2\alpha)}_j \zeta(q)\Big),
\end{equation}
where  $\zeta(q):=\sum_{n=1}^\infty n^{-q}$ is the Riemann zeta function. 
We further have, for any $\tau>0$,
    \begin{equation}\label{eqn:Adbound}
        |\Acal_d(L)|\leq 1+\frac{ L\exp(\tau^{-1})}{1+\tau\log {L} }P_{L}(\tau,d,\bsgamma),
    \end{equation}
    where $P_{L}(\tau,d,\bsgamma)$ is defined as in \eqref{eqn:PN}, but with $N\in\NN$ replaced by $L$.
\end{lemma}

\begin{proof} 
For $K\geq 0$, define
\[ 
\Acal^*_d(K) :=\left\{\bsh\in \NN^d : \prod_{j=1}^d h_j \leq  K\right\}.
\]
Observe that
\begin{align*}
    \Acal_d(L)&=\bigcup_{u\subseteq \{1{:}d\}}\left\{\bsh\in \ZZ^d \, \colon\, r_{1, \bsgamma^{1/(2\alpha)}}(\bsh) \leq L,h_j=0 \text{ if and only if }j\notin u\right\}\\
    &=\bigcup_{u\subseteq \{1{:}d\}}\left\{\bsh\in \ZZ^d \, \colon\,  \prod_{j\in u} |h_j| \leq L\,\prod_{j\in u} \gamma_j^{1/(2\alpha)},h_j=0 \text{ if and only if }j\notin u\right\},
\end{align*}
where the second equality follows from the assumption $1\ge \gamma_1 \ge \gamma_2 \ge \cdots$. 
Thus, we have
\begin{equation}\label{eqn:Addecompose}
    |\Acal_d(L)|=1+\sum_{\substack{u\subseteq \{1{:}d\}\\u\neq\emptyset}} 2^{|u|} \Big|\Acal^*_{|u|}\Big(L\,\prod_{j\in u} \gamma_j^{1/(2\alpha)}\Big)\Big|. 
\end{equation}
It suffices to bound $|\Acal^*_d(K)|$. When $d=1$, $|\Acal^*_d(K)|=\lfloor K\rfloor\leq K$. For $d\ge 2$, note that, given the values of $h_1,\dots,h_{d-1}$, there are $\Big\lfloor \Big(\prod_{j=1}^{d-1}h_j\Big)^{-1}K\Big\rfloor$ choices for $h_d$. 
Therefore, when $K\geq 1$,
\begin{align}\label{eqn:Adstarbound}
    |\Acal^*_d(K)|&=\sum_{h_1=1}^\infty \dots \sum_{h_{d-1}=1}^\infty \left\lfloor \left(\prod_{j=1}^{d-1}h_j\right)^{-1}K\right\rfloor\nonumber \\
    &\leq \inf_{q\geq 1} K^q\sum_{h_1=1}^\infty \dots \sum_{h_{d-1}=1}^\infty  \left(\prod_{j=1}^{d-1}h_j\right)^{-q}\,\bsone\left\{\prod_{j=1}^{d-1}h_j\leq K\right\}\nonumber\\
    &\leq \inf_{q\geq 1} K^q \prod_{j=1}^{d-1}\left(\sum_{h_j=1}^\infty h_j^{-q}\,\bsone\{h_j\leq K\}\right) \\
    &\leq \inf_{q> 1} K^q \zeta(q)^{d-1}, \nonumber
\end{align}
where $\bsone A$ denotes 
the indicator function of a set $A$. 
Substituting the above bound into \eqref{eqn:Addecompose}, we obtain
\begin{align}\label{eqn:Adboundsumu}
    |\Acal_d(L)|-1&\leq \inf_{q>1}  \sum_{\substack{u\subseteq \{1{:}d\}\\u\neq \emptyset}} 2^{|u|} \left(\prod_{j\in u} \gamma_j\right)^{q/(2\alpha)} L^q \zeta(q)^{|u|-1}\bsone\left\{L\prod_{j\in u} \gamma_j^{1/(2\alpha)} \ge 1\right\}\\
    & \leq   \inf_{q>1}  \frac{L^q}{\zeta(q)} \sum_{\substack{u\subseteq \{1{:}d\}\\u\neq \emptyset}} 2^{|u|} 
    \zeta(q)^{|u|} \prod_{j\in u} \gamma_j^{q/(2\alpha)} \nonumber\\
    &\le \inf_{q>1}  \frac{L^{q}}{\zeta(q)} \prod_{j=1}^d \Big(1+2\gamma^{q/(2\alpha)}_j \zeta(q)\Big).\nonumber
\end{align}
This completes the proof of \eqref{eqn:minoverq}.

Next, for $\tau>0$, we set
\[ q:=1+\frac{1}{\tau \log L} \]
in the infimum above to get an upper bound. It follows that
\[ L^{q}=\exp\left(\log L + \log L \frac{1}{\tau \log L}\right)=L\exp(\tau^{-1}). \]
Meanwhile, since
\[ \sum_{n=1}^\infty \frac{1}{n^q}\leq 1+\int_1^\infty \frac{1}{x^q}dx, \]
we obtain the bound
\[ \zeta(q)\leq 1+ \frac{1}{q-1}=1+\tau \log L. \]
Using this estimate and $\sup_{j\in \{1{:}d\}}\gamma_j\leq 1$, we can derive from \eqref{eqn:Adboundsumu} that
\begin{align*}
    |\Acal_d(L)|-1 &\leq L^q \sum_{\substack{u\subseteq \{1{:}d\}\\u\neq \emptyset}} 2^{|u|} 
    \zeta(q)^{|u|-1}  \prod_{j\in u} \gamma_j^{q/(2\alpha)} \\
    & \leq L\exp(\tau^{-1}) \sum_{\substack{u\subseteq \{1{:}d\}\\u\neq \emptyset}} 2^{|u|}  (1+\tau \log L)^{|u|-1}\prod_{j\in u} \gamma_j^{1/(2\alpha)}  \\
    & =\frac{ L\exp(\tau^{-1})}{1+\tau\log L}  \sum_{\substack{u\subseteq \{1{:}d\}\\u\neq \emptyset}} 2^{|u|} (1+\tau \log L)^{|u|} \prod_{j\in u} \gamma_j^{1/(2\alpha)} \\
    & \leq  \frac{ L\exp(\tau^{-1})}{1+\tau\log L}  \prod_{j=1}^d \Big(1+2\gamma^{1/(2\alpha)}_j (1+\tau\log L)\Big)\\
    &= \frac{ L\exp(\tau^{-1})}{1+\tau\log L} P_{L}(\tau,d,\bsgamma).
\end{align*}
This completes the proof of \eqref{eqn:Adbound}.
\end{proof}

\begin{corollary}\label{cor:AdNstarcount} 
    Let $\tau>0$.
    Suppose that $N$ is a prime number and let $N_*$ be defined as in \eqref{eqn:Nstar}. Then, whenever $N_*\geq 1$, it holds that 
    \[ |\Acal_d(N_*)|\leq 1+\frac{ N-1}{1+\tau\log N_*}.\]
\end{corollary}

\begin{proof}
    By the bound in \eqref{eqn:Adbound}, we have
    \begin{align*}
        |\Acal_d(N_*)|&\leq 1+ \frac{N_*\exp(\tau^{-1})}{1+\tau\log N_*} P_{N_*}(\tau,d,\bsgamma) \\
        & = 1+ \frac{N-1}{1+\tau\log N_*} \frac{P_{N_*}(\tau,d,\bsgamma)}{P_{N}(\tau,d,\bsgamma)} \\
        &\leq 1+ \frac{N-1}{1+\tau\log N_*},
    \end{align*}
    where the last inequality uses the facts that $N_*\leq N$ and that $P_{N}(\tau,d,\bsgamma)$ is increasing in $N$.
\end{proof}

Note that the condition $N_* \ge 1$ is precisely equivalent to
\begin{equation}\label{eq:cond_Nstar_one}
N \ge P_N (\tau, d, \bsgamma) \exp{(\tau^{-1})} +1.
\end{equation}

\begin{remark}
The bound \eqref{eqn:minoverq} can be improved by defining the partial sum $H_N(q):=\sum_{n=1}^N n^{-q}$ and applying the estimate
\[ |\Acal^*_d(K)|\leq \inf_{q\geq 1} K^q  H_{\lfloor K\rfloor}(q)^{d-1}\]
from \eqref{eqn:Adstarbound}. Although we omit the derivation details, it follows from this estimate (under the assumption $\sup_{j\in \{1{:}d\}}\gamma_j\leq 1$) that
\begin{align}\label{eqn:Adfinerbound}
    |\Acal_d(L)|-1 \leq \inf_{q\geq 1}  \frac{L^{q}}{ H_{L}(q)} \prod_{j=1}^d \Big(1+2\gamma^{q/(2\alpha)}_j  H_{L}(q)\Big).
\end{align}

The minimization in \eqref{eqn:Adfinerbound} can be carried out by solving for the zero of the logarithmic derivative of the right-hand side of \eqref{eqn:Adfinerbound}, i.e.,
\[ \log L -\frac{H'_L(q)}{H_L(q)}+\sum_{j=1}^d\frac{\alpha^{-1}\log(\gamma_j)\gamma^{q/(2\alpha)}_j  H_{L}(q)+2\gamma^{q/(2\alpha)}_j  H'_{L}(q)}{1+2\gamma^{q/(2\alpha)}_j  H_{L}(q)}=0. \]
Introducing the shorthand
\[ w_j(q)=\frac{2\gamma^{q/(2\alpha)}_j  H_{L}(q)}{1+2\gamma^{q/(2\alpha)}_j  H_{L}(q)},\]
this equation becomes
\begin{equation}\label{eqn:HNprime}
    -\frac{H'_L(q)}{H_L(q)}\left(\sum_{j=1}^d w_j(q)-1\right)-\frac{1}{2\alpha}\sum_{j=1}^d w_j(q)\log(\gamma_j)=\log L. 
\end{equation}

Note that $H'_L(q) < 0$ for all $q \geq 1$, and
\begin{align*}
   \left( \frac{H'_L(q)}{H_L(q)} \right)' &= \frac{H''_L(q)H_L(q) - \left(H'_L(q)\right)^2}{H_L(q)^2} \\
   &= \frac{1}{H_L(q)^2} \left( \sum_{n=1}^L \frac{(\log n)^2}{n^q} \sum_{n=1}^L \frac{1}{n^q} - \left( \sum_{n=1}^L \frac{\log n}{n^q} \right)^2 \right) > 0.
\end{align*}
Hence, $-H'_L(q)/H_L(q)$ is positive and strictly decreasing for $q \geq 1$. Furthermore, each $w_j(q)$ is strictly decreasing in $q$, and since $\log(\gamma_j) \leq 0$, the entire left-hand side of \eqref{eqn:HNprime} is strictly decreasing over $[1,\overline{q}]$, where $\overline{q}$ satisfies $\sum_{j=1}^d w_j(\overline{q})=1$. 

For the univariate case $d=1$, the global minimizer is $q=1$ as long as $\gamma^{1/2\alpha}_1 L>1$. When $d\ge 2$, the minimizer over $[1,\overline{q}]$ becomes the global minimizer for sufficiently large $L$. To show this, observe
\[ \lim_{q\to 1^+}\lim_{L\to\infty}\sum_{j=1}^d w_j(q)=\lim_{q\to 1^+}\sum_{j=1}^d \frac{2\gamma^{q/(2\alpha)}_j  \zeta(q)}{1+2\gamma^{q/(2\alpha)}_j  \zeta(q)}=d>1.\]
Hence for sufficiently small $\tau>0$, there exists $L_\tau \ge 1$ 
such that $\sum_{j=1}^d w_j(1+\tau)\geq 1$ for $L\geq L_\tau$. Using the estimate $H_L(1)\leq 1+ \log L$, we find
\[ \inf_{q\geq 1+\tau}  \frac{L^{q}}{ H_{L}(q)} \prod_{j=1}^d \left(1+2\gamma^{q/(2\alpha)}_j  H_{L}(q)\right)\geq \frac{L^{1+\tau}}{H_{L}(1)}\geq \frac{L^{1+\tau}}{1+ \log L}. \]
On the other hand, for $q=1+\tau/2$,
\[ \frac{L^{q}}{ H_{L}(q)} \prod_{j=1}^d \left(1+2\gamma^{q/(2\alpha)}_j  H_{L}(q)\right)\leq L^{1+\tau/2}\prod_{j=1}^d \left(1+2\gamma^{1/(2\alpha)}_j  (1+ \log L)\right), \]
which grows at a strictly slower rate than $L^{1+\tau}/(1+ \log L)$ as $L\to \infty$. Therefore, the global minimizer lies in $[1,1+\tau]\subset [1,\overline{q}]$ for large enough $L$. 

Although we rely on the bound \eqref{eqn:Adbound} for the subsequent analysis, one could alternatively use the improved bound \eqref{eqn:Adfinerbound} by choosing $q$ as the minimizer over $[1, \overline{q}]$. This minimization is numerically straightforward due to the monotonicity of the logarithmic derivative in \eqref{eqn:HNprime}.
\end{remark}

Let us proceed with our error analysis. To simplify notation, for a given $f\in\Hcal_{d,\alpha,\bsgamma}$, parameter $\tau>0$, and a prime number $N$, we write
\[ 
\epsilon(\bsh):=\left((\tau^{-1}+\log N_*) \left(\frac{\|f\|^2_{d,\alpha,\bsgamma}}{N_*^{2\alpha}(N-1) }+\sum_{\bsl \in N\cdot \ZZ^d\setminus \{\bszero\} }|\widehat{f}(\bsl+\bsh)|^2\right)\right)^{1/2},
\]
for $\bsh\in \ZZ^d$, where $N\cdot \ZZ^d=\{\bsl\in \ZZ^d : \ell_j \equiv 0 \tmod N \text{ for all } j\in \{1{:}d\}\}$. The following lemma establishes that, with some probability, each independent trial for estimating the Fourier coefficients corresponding to indices in the set $\Acal_d(N_*)$, i.e., $\widehat{f}_{N,\bsz_r,\bsDelta_r}(\bsh)$, for $r\in \{1{:}R\}$, yields a good approximation with the error bounded by $\epsilon(\bsh)$. 

\begin{lemma}\label{lem:concentration}
    Let $f\in\Hcal_{d,\alpha,\bsgamma}$, $\tau>0$, and let $N$ be prime such that $N_*\geq 1$.
    Using the notation introduced above, for any 
    $\bsh\in\Acal_d (N_*)$ and any $r\in\{1{:}R\}$, we have
    \[ 
    \Pr\Big(|\widehat{f}_{N,\bsz_r,\bsDelta_r}(\bsh)-\widehat{f}(\bsh)|^2>\epsilon(\bsh)^2\Big)
    \le  \frac{1+\tau}{1+\tau\log N_* }.
    \]
\end{lemma}

\begin{proof}
Fix $\bsh\in\Acal_d (N_*)$ and $r\in\{1{:}R\}$. Furthermore, we define
\[ P_{N,\boldsymbol{z}_r}^\perp:=\{\bsl\in \ZZ^d : \bsz^T_r\cdot\bsl\equiv 0 \tmod N\}.\]
Note that $N\cdot \ZZ^d \subset P_{N,\boldsymbol{z}_r}^\perp$ holds for any $N$ and $\bsz_r$.
Then, as outlined, for example, in \cite[Proof of Lemma 4.1]{CGK24}, we have
\begin{align}
    \int_{[0,1]^d} |\widehat{f}_{N,\bsz_r,\bsDelta_r}(\bsh)-\widehat{f}(\bsh)|^2 \rd\bsDelta_r &  = \int_{[0,1]^d} \abs{\sum_{\bsl \in P_{N,\boldsymbol{z}_r}^\perp\setminus \{\bszero\}}\widehat{f}(\bsl+\bsh) \exp(2\pi\icomp \bsl\cdot\bsDelta_r)}^2\rd\bsDelta_r\nonumber\\
    & =\sum_{\bsl \in P_{N,\boldsymbol{z}_r}^\perp\setminus \{\bszero\}}|\widehat{f}(\bsl+\bsh)|^2.\label{eqn:deltamean}
\end{align}
Moreover, for each $\bsl\in\ZZ^d\setminus (N\cdot \ZZ^d)$, we denote by $V_{\bsl}$ the event $\bsl \in P_{N,\boldsymbol{z}_r}^\perp$. 
Then, it follows from, for instance, \cite[Lemma~4]{kritzer2019lattice} that
\[ 
\Pr(V_{\bsl})=\Pr(\bsz^T_r\cdot\bsl\equiv 0 \tmod N )\leq \frac{1}{N-1}.
\]

We now define 
\[
\mathcal{B}_{\bsh}=\left\{ \bsl \in \ZZ^d\setminus (N\cdot \ZZ^d)\, \colon\, r_{2\alpha, \bsgamma}(\bsh+\bsl)\leq N_*^{2\alpha}\right\}.
\]
Then, by \Cref{cor:AdNstarcount}, we have
\begin{align*}
|\mathcal{B}_{\bsh}| & =
\left|\left\{ \bsl \in\ZZ^d\setminus (N\cdot \ZZ^d)\, \colon \, r_{2\alpha, \bsgamma}(\bsh + \bsl) \leq N_*^{2\alpha} \right\}\right|\\
& \le \left|\left\{ \bsl \in\ZZ^d\setminus \{\bszero\}\, \colon \, r_{2\alpha, \bsgamma}(\bsh + \bsl) \leq N_*^{2\alpha} \right\}\right| \\
& = \left|\left\{ \bsl \in \ZZ^d\setminus \{\bsh\} \, \colon \, r_{2\alpha, \bsgamma}(\bsl) \leq N_*^{2\alpha}\right\}\right|\\
&\le \frac{N-1}{1+\tau\log N_*},
\end{align*}
where we have used the fact that $r_{2\alpha, \bsgamma}(\bsh) \leq N_*^{2\alpha}$ for $\bsh\in \Acal_d(N_*)$.

Therefore, the union bound argument shows
\begin{align}\label{eqn:dualspace}
\Pr(P_{N,\boldsymbol{z}_r}^\perp\cap \mathcal{B}_{\bsh} \neq \emptyset)
&= 
\Pr\left(\bigcup_{\bsl\in\mathcal{B}_{\bsh}} V_{\bsl} \right)\nonumber\\
&\le \sum_{\bsl\in\mathcal{B}_{\bsh}} \Pr (V_{\bsl}) \nonumber\\
&\le \frac{N-1}{1+\tau\log N_*}\cdot \frac{1}{N-1}\nonumber\\
&= \frac{1}{1+\tau\log N_*}.
\end{align}
Now we apply Markov's inequality to obtain
\begin{align}\label{eqn:Markov}
    & \Pr\Big(|\widehat{f}_{N,\bsz_r,\bsDelta_r}(\bsh)-\widehat{f}(\bsh)|^2>\epsilon(\bsh)^2\Big)\nonumber\\
    & = \Pr\Big(\{|\widehat{f}_{N,\bsz_r,\bsDelta_r}(\bsh)-\widehat{f}(\bsh)|^2> \epsilon(\bsh)^2\}\cap \{P_{N,\boldsymbol{z}_r}^\perp\cap \mathcal{B}_{\bsh}\neq \emptyset\}\Big) \nonumber\\
    & \quad +\Pr\Big(\{|\widehat{f}_{N,\bsz_r,\bsDelta_r}(\bsh)-\widehat{f}(\bsh)|^2> \epsilon(\bsh)^2\}\cap \{P_{N,\boldsymbol{z}_r}^\perp\cap \mathcal{B}_{\bsh}= \emptyset\}\Big) \nonumber\\
    & \leq \Pr(P_{N,\boldsymbol{z}_r}^\perp\cap \mathcal{B}_{\bsh}\neq \emptyset) \nonumber\\
    & \quad +\Pr\Big(\{\epsilon(\bsh)^{-2}|\widehat{f}_{N,\bsz_r,\bsDelta_r}(\bsh)-\widehat{f}(\bsh)|^2>1\}\cap \{\bsone\{P_{N,\boldsymbol{z}_r}^\perp\cap \mathcal{B}_{\bsh}= \emptyset\}= 1\}\Big) \nonumber\\
    & \leq \Pr(P_{N,\boldsymbol{z}_r}^\perp\cap \mathcal{B}_{\bsh}\neq \emptyset)\nonumber\\
    & \quad +\Pr\Big(\epsilon(\bsh)^{-2}|\widehat{f}_{N,\bsz_r,\bsDelta_r}(\bsh)-\widehat{f}(\bsh)|^2\bsone\{P_{N,\boldsymbol{z}_r}^\perp\cap \mathcal{B}_{\bsh}= \emptyset\}> 1\Big) \nonumber\\
    & \leq \frac{1}{1+ \tau\log N_*} 
    +\frac{1}{\epsilon(\bsh)^2}\e\Big[|\widehat{f}_{N,\bsz_r,\bsDelta_r}(\bsh)-\widehat{f}(\bsh)|^2\bsone\{P_{N,\boldsymbol{z}_r}^\perp\cap \mathcal{B}_{\bsh}= \emptyset\}\Big].
\end{align}
Here, the probability $\Pr$ and expectation $\e$ are taken over the random choices of $\bsz_r$ and $\bsDelta_r$.
From the equality \eqref{eqn:deltamean}, the inclusion $N\cdot \ZZ^d \subset P_{N,\boldsymbol{z}_r}^\perp$, and the fact that $\bsDelta_r$ is independent of $\bsz_r$, we have
\begin{align*}
    &\e\Big[|\widehat{f}_{N,\bsz_r,\bsDelta_r}(\bsh)-\widehat{f}(\bsh)|^2\bsone\{P_{N,\boldsymbol{z}_r}^\perp\cap \mathcal{B}_{\bsh}= \emptyset\}\Big]\\
    & = \frac{1}{(N-1)^d}\sum_{\bsz_r\in \{1{:}(N-1)\}^d}\sum_{\bsl \in P_{N,\boldsymbol{z}_r}^\perp\setminus \{\bszero\}}|\widehat{f}(\bsl+\bsh)|^2 \bsone\{P_{N,\boldsymbol{z}_r}^\perp\cap \mathcal{B}_{\bsh}= \emptyset\}\\
    & \leq \frac{1}{(N-1)^d}\sum_{\bsz_r\in \{1{:}(N-1)\}^d}\sum_{\bsl \in P_{N,\boldsymbol{z}_r}^\perp\setminus (N\cdot \ZZ^d) }|\widehat{f}(\bsl+\bsh)|^2 \frac{r_{2\alpha, \bsgamma}(\bsh+\bsl) }{N_*^{2\alpha}}\\
    & \quad +\frac{1}{(N-1)^d}\sum_{\bsz_r\in \{1{:}(N-1)\}^d}\sum_{\bsl \in N\cdot \ZZ^d\setminus \{\bszero\} }|\widehat{f}(\bsl+\bsh)|^2 \\
    & =\sum_{\bsl \in \ZZ^d\setminus (N\cdot \ZZ^d)}|\widehat{f}(\bsl+\bsh)|^2\frac{r_{2\alpha, \bsgamma}(\bsh+\bsl) }{N_*^{2\alpha}}\left(\frac{1}{(N-1)^d}\sum_{\bsz_r\in \{1{:}(N-1)\}^d}\bsone\{\bsl \in P_{N,\boldsymbol{z}_r}^\perp\}\right)\\
    & \quad +\sum_{\bsl \in N\cdot \ZZ^d\setminus \{\bszero\} }|\widehat{f}(\bsl+\bsh)|^2\\
    & = \frac{1}{N_*^{2\alpha}}\,\sum_{\bsl \in \ZZ^d\setminus (N\cdot \ZZ^d)}|\widehat{f}(\bsl+\bsh)|^2 \, r_{2\alpha, \bsgamma}(\bsh+\bsl)\,  \Pr(V_{\bsl})+\sum_{\bsl \in N\cdot \ZZ^d\setminus \{\bszero\} }|\widehat{f}(\bsl+\bsh)|^2\\
    & \leq  \frac{\|f\|^2_{d,\alpha,\bsgamma}}{N_*^{2\alpha}(N-1) }+\sum_{\bsl \in N\cdot \ZZ^d\setminus \{\bszero\} }|\widehat{f}(\bsl+\bsh)|^2.
\end{align*}
The conclusion follows once we substitute the above upper bound into \eqref{eqn:Markov}.
\end{proof}

The next lemma is about the probability amplification via the median trick. That is, in the previous lemma, we showed that each individual estimate is accurate with some probability. By applying the median trick to the aggregate estimate $\widehat{f}_{N,\bsz_{\{1{:}R\}},\bsDelta_{\{1{:}R\}}}(\bsh)$, we can enhance the overall success probability as the number of repetitions $R$ grows.

\begin{lemma}\label{lem:medianbound} 
    Let $f\in\Hcal_{d,\alpha,\bsgamma}$, $\tau>0$, and let $N$ be prime such that $N_*\geq 1$. Suppose that $R$ is an odd number. Then, for any $\bsh\in\Acal_d (N_*)$, we have
        \[ 
        \Pr\Big(|\widehat{f}_{N,\bsz_{\{1{:}R\}},\bsDelta_{\{1{:}R\}}}(\bsh)-\widehat{f}(\bsh)|^2> 2\epsilon(\bsh)^2\Big)
        \le  \,  \left(\frac{4(1+\tau)}{1+\tau\log N_* }\right)^{\lceil R/2 \rceil}.
        \]
\end{lemma}
\begin{proof} 
    By \Cref{lem:concentration}, for each $r\in \{1{:}R\}$, we have
    \begin{eqnarray*}
    \Pr\Big(|\Re\widehat{f}_{N,\bsz_r,\bsDelta_r}(\bsh)-\Re\widehat{f}(\bsh)|^2> \epsilon(\bsh)^2\Big)
    &\le& \frac{1+\tau}{1+\tau\log N_* }. 
    \end{eqnarray*}
    In order for $|\Re\widehat{f}_{N,\bsz_{\{1{:}R\}},\bsDelta_{\{1{:}R\}}}(\bsh)-\Re\widehat{f}(\bsh)|^2>\epsilon(\bsh)^2$ to hold, at least $\lceil R/2\rceil$ instances of the $R$ estimates must satisfy $|\Re\widehat{f}_{N,\bsz_r,\bsDelta_r}(\bsh)-\Re\widehat{f}(\bsh)|^2> \epsilon(\bsh)^2$. There are at most $2^R$ distinct subsets of $\{1{:}R\}$, so we obtain
    \begin{align*}
        \Pr\Big(|\Re\widehat{f}_{N,\bsz_{\{1{:}R\}},\bsDelta_{\{1{:}R\}}}(\bsh)-\Re\widehat{f}(\bsh)|^2> \epsilon(\bsh)^2\Big)
         &\le 2^R \,  \left(\frac{1+\tau}{1+\tau\log N_* } \right)^{\lceil R/2 \rceil}\\
        & = \,  \frac{1}{2}\left(\frac{4(1+\tau)}{1+\tau\log N_* }\right)^{\lceil R/2 \rceil},
    \end{align*}
    where we used the fact that $R$ is odd in the last step.
    A similar bound holds for the imaginary part, and the conclusion follows by applying the union bound.     
\end{proof}

We now arrive at the main theoretical result of this paper.
\begin{theorem}\label{thm:L2rate_constants}
    Let $\tau>0$ and $\delta \in (0,1)$ be given, and let $N$ be a prime sufficiently large that $1\leq N_*< N/2$ and
    \begin{equation}\label{eq:less_than_one}
    \frac{4(1+\tau)}{1+\tau\log N_* }<1.
    \end{equation}
    Moreover, let $R$ be an odd integer sufficiently large such that
    \begin{equation}\label{eq:choose_R_first}
    \Big(1+\frac{N-1}{1+\tau\log N_*}\Big)
    \left(\frac{4(1+\tau)}{1+\tau\log N_* }\right)^{\lceil R/2 \rceil}\le \delta.
    \end{equation}
    Then, with probability at least $1-\delta$, we have that for any $f\in \Hcal_{d,\alpha,\bsgamma}$,
    \[ \|f-A^{\rand}_{N,\bsz_{\{1{:}R\}},\bsDelta_{\{1{:}R\}},\Acal_d(N_*)}(f)\|^2_{L_2}\le \frac{\|f\|^2_{d,\alpha,\bsgamma}}{N^{2\alpha}_*}\Big(2\tau^{-1}+1 + \frac{2N}{N-1}\log(N-1) \Big).\]
\end{theorem}

It is worth noting that the probability that the error bound in \Cref{thm:L2rate_constants} fails decreases exponentially with increasing $R$.

\begin{proof}[Proof of \Cref{thm:L2rate_constants}]
    It follows from \Cref{cor:AdNstarcount}, \Cref{lem:medianbound}, and the assumption in the theorem that
    \begin{align*}
    & \Pr\Big(|\widehat{f}_{N,\bsz_{\{1{:}R\}},\bsDelta_{\{1{:}R\}}}(\bsh)-\widehat{f}(\bsh)|^2> 2\epsilon(\bsh)^2 \text{ for at least one }  \bsh\in \Acal_d(N_*)\Big)\\  
    &\le \sum_{\bsh\in \Acal_d(N_*)}\Pr\Big(|\widehat{f}_{N,\bsz_{\{1{:}R\}},\bsDelta_{\{1{:}R\}}}(\bsh)-\widehat{f}(\bsh)|^2> 2\epsilon(\bsh)^2\Big) \\
    &\le |\Acal_d(N_*)| \, \,  \left(\frac{4(1+\tau)}{1+\tau\log N_* }\right)^{\lceil R/2 \rceil} \\
    &\le  
    \Big(1+\frac{N-1}{1+\tau\log N_*}\Big)
    \left(\frac{4(1+\tau)}{1+\tau\log N_* }\right)^{\lceil R/2 \rceil} \le \delta.
    \end{align*}
    Therefore, the aggregated estimate satisfies $|\widehat{f}_{N,\bsz_{\{1{:}R\}},\bsDelta_{\{1{:}R\}}}(\bsh)-\widehat{f}(\bsh)|^2\leq 2\epsilon(\bsh)^2$ simultaneously for all $\bsh\in \Acal_d(N_*)$, with probability at least $1-\delta$.
    
    When this event occurs, by Parseval's identity, we have
    \begin{align*}
        &\|f-A^{\rand}_{N,\bsz_{\{1{:}R\}},\bsDelta_{\{1{:}R\}},\Acal_d(N_*)}(f)\|^2_{L_2}\\
        & = \sum_{\bsh\notin \mathcal{A}_d(N_*)} |\widehat{f}(\bsh)|^2+\sum_{\bsh\in \Acal_d(N_*)}|\widehat{f}_{N,\bsz_{\{1{:}R\}},\bsDelta_{\{1{:}R\}}}(\bsh)-\widehat{f}(\bsh)|^2\\
        & \le \sum_{\bsh\notin \mathcal{A}_d(N_*)} |\widehat{f}(\bsh)|^2+2\sum_{\bsh\in \Acal_d(N_*)}\epsilon(\bsh)^2\\
        & \le \sum_{\bsh\notin \mathcal{A}_d(N_*)} |\widehat{f}(\bsh)|^2 
        + 2 |\mathcal{A}_d(N_*)|\,\frac{\tau^{-1}+\log N_*}{N-1}\,\frac{\|f\|^2_{d,\alpha,\bsgamma}}{N_*^{2\alpha}}\\
        & \quad +2(\tau^{-1}+\log N_*)\sum_{\bsh\in \Acal_d(N_*)}\sum_{\bsl \in N\cdot \ZZ^d\setminus \{\bszero\} }|\widehat{f}(\bsl+\bsh)|^2.
    \end{align*}

    Again, by \Cref{cor:AdNstarcount}, it holds that
    \[ |\mathcal{A}_d(N_*)|\,\frac{\tau^{-1}+\log N_*}{N-1}\leq \frac{\tau^{-1}+\log N_*}{N-1}+\tau^{-1}.\]
    Since $P_N (\tau,d,\bsgamma)\geq 1$ and $N_*=(N-1)/(P_N (\tau,d,\bsgamma)\exp{(\tau^{-1})})$, we have
    \[ \tau^{-1}+\log N_*= \log (N-1)-\log(P_N (\tau,d,\bsgamma))\leq \log (N-1).\]
    
    Finally, $\bsh\in \mathcal{A}_d(N_*)$ implies $\bsh\in (-N/2,N/2)^d$ when $N_*<N/2$, so the shifted lattices $\bsh + N\cdot \ZZ^d$ are mutually disjoint for different $\bsh\in \mathcal{A}_d(N_*)$. Also, it is clear that for any $\bsh\in \mathcal{A}_d(N_*)$ and $\bsl \in N\cdot \ZZ^d\setminus \{\bszero\}$, we have $\bsl+\bsh\not\in \mathcal{A}_d(N_*)$. Therefore, 
    \begin{align*}
        \sum_{\bsh\in \Acal_d(N_*)}\sum_{\bsl \in N\cdot \ZZ^d\setminus \{\bszero\} }|\widehat{f}(\bsl+\bsh)|^2 & \leq \sum_{\bsh\notin \mathcal{A}_d(N_*)} |\widehat{f}(\bsh)|^2\\ & \leq \sum_{\bsh\notin \mathcal{A}_d(N_*)} |\widehat{f}(\bsh)|^2 \frac{r_{2\alpha, \bsgamma}(\bsh) }{N^{2\alpha}_*} \leq \frac{\|f\|^2_{d,\alpha,\bsgamma}}{N^{2\alpha}_*}.
    \end{align*}
    Hence,
    \begin{align*}
        & \|f-A^{\rand}_{N,\bsz_{\{1{:}R\}},\bsDelta_{\{1{:}R\}},\Acal_d(N_*)}(f)\|^2_{L_2}\\
        & \leq \frac{\|f\|^2_{d,\alpha,\bsgamma}}{N^{2\alpha}_*}+2\left(\tau^{-1}+\frac{\log(N-1)}{N-1}\right) \frac{\|f\|^2_{d,\alpha,\bsgamma}}{N^{2\alpha}_*} + 2\log(N-1) \frac{\|f\|^2_{d,\alpha,\bsgamma}}{N^{2\alpha}_*} 
    \end{align*}
    and our conclusion follows.
\end{proof}

\begin{remark}\label{rmk:choosingRN}
  Let us briefly discuss the assumptions in \Cref{thm:L2rate_constants}.
  
  \begin{enumerate}
  \item Note that \eqref{eq:less_than_one} is certainly satisfied, due to the definitions of $N_*$ and $P_N(\tau,d,\bsgamma)$, when
  \[ 
    \frac{4(1+\tau)}{\tau \log (N-1)-\tau\log (P_{N} (\tau,d,\bsgamma)) } \le c,
  \]
  for some constant $c \in (0,1)$, which in turn holds if
  \begin{equation}\label{eq:less_than_c}
  N \ge P_N (\tau,d,\bsgamma)\, \exp \left(\frac{4\tau^{-1}+4 }{c} \right) + 1.
  \end{equation}

   \item If $N$ is chosen according to \eqref{eq:less_than_c}, then \eqref{eq:cond_Nstar_one} is also satisfied, ensuring $N_* \ge 1$. Moreover, since
  \[ 
  N_* = \frac{N-1}{\exp(\tau^{-1})P_{N}(\tau,d,\bsgamma)} < \frac{N}{\exp(\tau^{-1})\left(1+2\sum_{j=1}^d \gamma_j^{1/(2\alpha)}(1+\tau \log N)\right)},
  \]
  we have $N_* < N/2$ for all prime $N$ as long as $\tau^{-1} \ge \log 2$ or $\sum_{j=1}^d \gamma_j^{1/(2\alpha)} \ge 1/2$. Since this condition can easily be satisfied by choosing either $\tau \le 1/\log 2$ or $\gamma_1 = 1$, we will not emphasize this assumption in the subsequent analysis.

  \item The condition \eqref{eq:choose_R_first} is fulfilled if
  \begin{multline*}
   \log \left(1+\frac{N-1}{\tau\log (N-1) - \tau\log (P_{N} (\tau,d,\bsgamma)) }\right) 
   \\ + \left\lceil \frac{R}{2} \right\rceil \log \left(\frac{4(1+\tau)}{\tau\log (N-1) - \tau\log (P_{N} (\tau, d,\bsgamma)) }\right) \le \log \delta.
  \end{multline*} 
  Indeed, suppose $c \in (0,1)$ and $N$ is chosen according to \eqref{eq:less_than_c}. Then it suffices for the odd integer $R$ to satisfy 
  \[
   \log \left(1 + \frac{c (N-1)}{4}\right)
   + \left\lceil \frac{R}{2} \right\rceil \log c \le \log \delta,
  \]
  or equivalently,
  \begin{equation}\label{eq:choose_R_c}
  \frac{R+1}{2} \log \left(c^{-1}\right) \ge
  \log \left(1 + \frac{c (N-1)}{4}\right) + \log(\delta^{-1}).
  \end{equation}
  \end{enumerate}
\end{remark}

Next, we state an important corollary of \Cref{thm:L2rate_constants}, which implies that---apart from logarithmic factors---\Cref{alg:median} achieves the optimal convergence rate (in terms of the number of function evaluations of a given function~$f$) for $L_2$-approximation, with high probability and for individual functions in $\Hcal_{d,\alpha,\bsgamma}$. 

\begin{corollary}\label{cor:error_bound}
 Let $\tau>0$ and $\delta \in (0,1)$ be given.
 Denote by $M = RN$ the total number of evaluations of $f$ required by the algorithm $A^{\rand}_{N,\bsz_{\{1{:}R\}},\bsDelta_{\{1{:}R\}},\Acal_d(N_*)}$.
 Then, there exist choices of a prime number $N$ and an odd integer $R$ such that for any $f \in \Hcal_{d,\alpha,\bsgamma}$, with probability at least $1 - \delta$,
 \begin{align*}
  \|f - A^{\rand}_{N,\bsz_{\{1{:}R\}},\bsDelta_{\{1{:}R\}},\Acal_d(N_*)}(f)\|^2_{L_2}
  \le C_N(\tau,\delta,\alpha)\, \frac{\|f\|^2_{d,\alpha,\bsgamma}}{M^{2\alpha}}
  \exp(2\alpha\tau^{-1})\, (P_N(\tau,d,\bsgamma))^{2\alpha},
 \end{align*}
 where
 \begin{align*}
  C_N(\tau,\delta,\alpha) &=
  \left(\frac{N}{N-1}\right)^{2\alpha}
  \left(2\log\left(1 + \frac{N-1}{4e}\right) + 2\log(\delta^{-1}) + 1\right)^{2\alpha} \\
  & \quad \times \left(2\tau^{-1} + 1 + \frac{2N}{N-1}\log(N-1)\right).
 \end{align*}
\end{corollary}

\begin{proof}
 Choose $c := \exp(-1)$, and a prime $N$ such that \eqref{eq:less_than_c} holds for this choice of $c$. 
 This is possible for sufficiently large $N$ because $P_N(\tau,d,\bsgamma)$ depends only logarithmically on $N$. Furthermore, with $c = \exp(-1)$, \eqref{eq:choose_R_c} is certainly satisfied if we choose the odd integer $R$ such that
    \begin{equation}\label{eq:choose_R_coro}
        2\log \left(1 + \frac{N-1}{4e}\right)+2\log(\delta^{-1})-1
        \le R \le 
        2\log \left(1 + \frac{N-1}{4e}\right)+2\log(\delta^{-1})+1.
    \end{equation}
  Next, we apply \Cref{thm:L2rate_constants} for these choices of $N$ and $R$ and obtain, with probability at least $1 - \delta$,
  \begin{align*}
  &\|f-A^{\rand}_{N,\bsz_{\{1{:}R\}},\bsDelta_{\{1{:}R\}},\Acal_d(N_*)}(f)\|^2_{L_2} \\
  &\le
  \frac{\|f\|^2_{d,\alpha,\bsgamma}}{N^{2\alpha}_*}\Big(2\tau^{-1}+1 + \frac{2N\log(N-1)}{N-1} \Big)\\
  &\le 
  \frac{\|f\|^2_{d,\alpha,\bsgamma}}{(N-1)^{2\alpha}}
  \,\Big(2\tau^{-1}+1 + \frac{2N\log(N-1)}{N-1} \Big)\, \exp(2\alpha\tau^{-1})\,
  (P_N (\tau,d,\bsgamma))^{2\alpha}\\
  &=
  \left(\frac{RN}{N-1}\right)^{2\alpha} \frac{\|f\|^2_{d,\alpha,\bsgamma}}{M^{2\alpha}}\,\Big(2\tau^{-1}+1 + \frac{2N\log(N-1)}{N-1} \Big)\,
  \exp(2\alpha\tau^{-1})\,
  (P_N (\tau,d,\bsgamma))^{2\alpha}.
  \end{align*}
  The result of the corollary now follows by replacing $R$ with the upper bound in \eqref{eq:choose_R_coro}.
\end{proof}

In the light of the $(\epsilon,\delta)$-approximation framework \eqref{eq:eps-delta_framework}, \Cref{cor:error_bound} directly shows that we have
\[ \epsilon = \frac{1}{M^{\alpha}}(C_N(\tau,\delta,\alpha))^{1/2}
  \exp(\alpha\tau^{-1})\, (P_N(\tau,d,\bsgamma))^{\alpha}\]
for any $\delta\in (0,1)$.
We remark that the upper bound obtained in \Cref{cor:error_bound}, along with the choices of $N$ and $R$, generally depends on the dimension $d$. This issue will be addressed in the following section.

\section{Dependence of the error bound on the dimension}\label{sec:tractability}
In this section, let us return, for given $f\in\Hcal_{d,\alpha,\bsgamma}$, to the approximation error of our median lattice-based algorithm,
\[
  \mathrm{err}(f,L_2, A^{\rand}_{N,\bsz_{\{1{:}R\}},\bsDelta_{\{1{:}R\}},\Acal_d(N_*)}):=\|f-A^{\rand}_{N,\bsz_{\{1{:}R\}},\bsDelta_{\{1{:}R\}},\Acal_d(N_*)}(f)\|_{L_2},
\]
and examine under which circumstances we can overcome the curse of dimensionality, depending on the decay of the weights $\bsgamma$.

Let $f\in\Hcal_{d,\alpha,\bsgamma}$ be given and fixed.
By \Cref{cor:error_bound}, for a given $\tau > 0$ and $\delta \in (0,1)$, we can choose $R$ and $N$ such that, with probability at least $1 - \delta$,
\begin{align*}
& [\mathrm{err}(f, L_2, A^{\rand}_{N,\bsz_{\{1{:}R\}},\bsDelta_{\{1{:}R\}},\Acal_d(N_*)})]^2 \\
&\qquad \le \norm{f}_{d,\alpha,\bsgamma}^2
\frac{C_N(\tau,\delta,\alpha) }{M^{2\alpha}} 
    \exp(2\alpha\tau^{-1})
    (P_N (\tau,d,\bsgamma))^{2\alpha} .
\end{align*}
We now choose $\tau>0$, such that
\[
  \tau^{-1}:=\eta^{-1}\, G_d,\quad \mbox{where}\quad G_d:= 2\sum_{j=1}^d \gamma_j^{1/(2\alpha)},
\]
for some absolute constant $\eta>0$. With this choice of $\tau$, we obtain
\begin{align*}
P_N (\tau,d,\bsgamma) &= 
\prod_{j=1}^d \left(1+ 2\gamma_j^{1/(2\alpha)} (1+\eta\,G_d^{-1}\log N)\right)\\
&= 
   \exp \left(\sum_{j=1}^d \log \left(1+ 2\gamma_j^{1/(2\alpha)} (1+\eta\,G_d^{-1}\log N)\right)\right)\\
&\le 
   \exp \left(\sum_{j=1}^d 2\gamma_j^{1/(2\alpha)} (1+\eta\,G_d^{-1}\log N)\right)\\
&= \exp(\, G_d) \, N^\eta.
\end{align*}
Hence, for a given $\delta \in (0,1)$, we can choose $R$ and $N$ such that, with probability at least $1 - \delta$,
\begin{align}\label{eq:err_bound}
       & [\mathrm{err}(f, L_2, A^{\rand}_{N,\bsz_{\{1{:}R\}},\bsDelta_{\{1{:}R\}},\Acal_d(N_*)})]^2 \notag \\
& \qquad \le \norm{f}_{d,\alpha,\bsgamma} \frac{C_N( G^{-1}_d \eta ,\delta,\alpha) }{M^{2\alpha}} \exp(2\alpha(\eta^{-1}+1)\, G_d)  N^{2\alpha\eta}. 
\end{align}

Let us consider two special cases regarding the decay of the weights $\bsgamma$.

\begin{case}
Suppose that $\sum_{j=1}^\infty \gamma_j^{1/(2\alpha)} < \infty$. 

In this case, we define $G_\infty:=2\sum_{j=1}^\infty \gamma_j^{1/(2\alpha)}$. Then, for all $d \in \NN$ we have $G_d \le G_\infty$, and the bound \eqref{eq:less_than_c} with $c = \exp(-1)$ is satisfied if
\[
 N \ge \exp\left(\, (4e\eta^{-1}+1)G_\infty+4e\right) \, N^\eta +1,
\]
which holds for sufficiently large $N$, independently of $d$, provided that $\eta \in (0,1)$. Moreover, with this choice of $N$, we can select an odd $R$ according to \eqref{eq:choose_R_coro}, again independently of $d$.

Finally, using $G_d \le G_\infty$ in \eqref{eq:err_bound}, we find that the right-hand side of \eqref{eq:err_bound} can be bounded by a quantity that is independent of $d$, with a convergence rate that can be made arbitrarily close to $M^{-2\alpha}$ by choosing $\eta$ sufficiently small.
\end{case}

\begin{case}
Suppose that ${\displaystyle \limsup_{d\to\infty} \left(\sum_{j=1}^d \gamma_j^{1/(2\alpha)}\right)/\log d < \infty}$.

Then, there exists a constant $D > 0$ such that $G_d \le D \log d$ for sufficiently large $d$. 
In this setting, the condition \eqref{eq:less_than_c} with $c = \exp(-1)$ is satisfied if
\[
 N \ge \exp\left(\, (4e\eta^{-1}+1)G_d+4e\right) \, N^\eta+1,
\]
which can be achieved by choosing $N$ that grows at most polynomially with $d$. Using this choice of $N$, we can also choose an odd $R$ satisfying \eqref{eq:choose_R_coro}, which then depends at most logarithmically on $d$.

Finally, from the error bound \eqref{eq:err_bound}, we observe that the dependence on $G_d$ leads to a bound that depends at most polynomially on $d$. Moreover, the convergence rate can be made arbitrarily close to $M^{-2\alpha}$ by selecting $\eta$ sufficiently small.
\end{case}

We summarize our findings in the following theorem.
\begin{theorem}\label{thm:error_bound}
    Let $f\in\Hcal_{d,\alpha,\bsgamma}$ and $\delta \in (0,1)$ be given. Denote by $M = RN$ the total number of function evaluations used by the algorithm $A^{\rand}_{N,\bsz_{\{1{:}R\}},\bsDelta_{\{1{:}R\}},\Acal_d(N_*)}$. Then for any $\eta \in (0,1)$, the following statements hold:
    \begin{itemize}
    \item If $\sum_{j=1}^\infty \gamma_j^{1/(2\alpha)} < \infty$, then there exist a prime $N$ and an odd $R$, both independent of $d$ (with $R$ depending logarithmically on $N$), such that with probability at least $1 - \delta$, 
    \begin{align*}
     & [\mathrm{err}(f, L_2, A^{\rand}_{N,\bsz_{\{1{:}R\}},\bsDelta_{\{1{:}R\}},\Acal_d(N_*)})]^2\\
    & \qquad \le \norm{f}_{d,\alpha,\bsgamma}\frac{C_N( G^{-1}_\infty \eta ,\delta,\alpha) }{M^{2\alpha}}
    \exp(2\alpha(\eta^{-1}+1) G_\infty)  N^{2\alpha\eta}.
    \end{align*}
    This means that the error bound for the $L_2$-approximation problem using $ A^{\rand}_{N,\bsz_{\{1{:}R\}},\bsDelta_{\{1{:}R\}},\Acal_d(N_*)}$ does not depend on the dimension $d$.

    \item If $\displaystyle \limsup_{d \to \infty} \left(\sum_{j=1}^d \gamma_j^{1/(2\alpha)}\right)/\log d < \infty$, then there exist a prime $N$ that depends at most polynomially on $d$, and an odd $R$ that depends at most logarithmically on $d$ and $N$, such that with probability at least $1 - \delta$,
    \begin{align*}
    & [\mathrm{err}(f, L_2, A^{\rand}_{N,\bsz_{\{1{:}R\}},\bsDelta_{\{1{:}R\}},\Acal_d(N_*)})]^2 \\
    & \qquad \le \norm{f}_{d,\alpha,\bsgamma}\frac{C_N( G^{-1}_d \eta ,\delta,\alpha) }{M^{2\alpha}}\,
    d^{2\alpha D(\eta^{-1}+1)} N^{2\alpha\eta},
    \end{align*}
    for some constant $D > 0$ independent of $d$ and $N$.
    This means that the error bound for the $L_2$-approximation problem using $ A^{\rand}_{N,\bsz_{\{1{:}R\}},\bsDelta_{\{1{:}R\}},\Acal_d(N_*)}$ depends only polynomially on $d$.
\end{itemize}
\end{theorem}

\begin{remark}
    When  $\sum_{j=1}^\infty \gamma_j^{1/(2\alpha)} < \infty$, a slight modification of \cite[Lemma 3]{hickernell2003existence} shows that, for any  $\tau,\eta>0$, there exists a constant $C(\bsgamma, \alpha, \tau, \eta)$, independent of both $d$ and $N$, such that
    \[
    P_N (\tau,d,\bsgamma) \le C(\bsgamma, \alpha, \tau, \eta) N^{\eta}.
    \]
    In particular, this implies that the $L_2$-approximation problem using $ A^{\rand}_{N,\bsz_{\{1{:}R\}},\bsDelta_{\{1{:}R\}},\Acal_d(N_*)}$ has an error bound that is independent of $d$ regardless of our choice of $\tau$.
\end{remark}

\begin{remark}
    The lower bound on $N$ derived in this section is sub-optimal and is used solely for the purpose of analyzing the dependence on the dimension. For practical implementation, one should refer to \Cref{rmk:choosingRN} for guidance on selecting $N$.
\end{remark}

\begin{remark}\label{rmk:choosingtau}
    As suggested by the above analysis, one should choose $\tau$ proportional to $G_d^{-1}$ to avoid the curse of dimensionality, particularly when $\sum_{j=1}^d \gamma_j^{1/(2\alpha)}$ diverges as $d \to \infty$.

    Given a computational budget $M_{\max}$ for $M = RN$, one may choose $\tau$ using the following procedure. First, estimate $R$ via \eqref{eq:choose_R_coro}, and let $N_{\max}$ be the largest prime $N$ such that
    \begin{equation}\label{eq:maximum_prime}
    N\left(2\log \left(1 + \frac{N-1}{4e}\right)+2\log(\delta^{-1})+1\right)\leq M_{\max}.
    \end{equation}
    To ensure \eqref{eq:less_than_c} holds with $N = N_{\max}$ and $c = \exp(-1)$, we require
    \[ 
    \inf_{\tau>0}\exp \left(4e\tau^{-1}\right)\,P_{N_{\max}} (\tau,d,\bsgamma)\leq \exp(-4e)\left(N_{\max}-1\right).
    \]
    Since
    \begin{equation}\label{eq:logderiative1}
        \tau\frac{d}{d\tau}\log\left(\exp \left(4e\tau^{-1}\right)P_{N_{\max}} (\tau,d,\bsgamma)\right)=-\frac{4e}{\tau}+\sum_{j=1}^d \frac{2\gamma_j^{1/(2\alpha)}\tau \log N_{\max} }{1+ 2\gamma_j^{1/(2\alpha)} (1+\tau\log N_{\max})}
    \end{equation}
    is strictly increasing in $\tau>0$, it has a unique zero $\tau_0$ and we have
    \[ 
    \inf_{\tau>0}\exp \left(4e\tau^{-1}\right)P_{N_{\max}} (\tau,d,\bsgamma)=\exp \left(4e{\tau_0}^{-1}\right)P_{N_{\max}} (\tau_0,d,\bsgamma).
    \]
    We check whether \eqref{eq:less_than_c} holds with $\tau = \tau_0$, $N = N_{\max}$, and $c = \exp(-1)$. If not, this suggests that $M_{\max}$ is too small and must be increased.
    For sufficiently large $M_{\max}$ and the corresponding $N_{\max}$, the equation
    \begin{equation*}
     \exp \left(4e\tau^{-1}\right) P_{N_{\max}} (\tau,d,\bsgamma)=\exp(-4e)\left(N_{\max}-1\right)   
    \end{equation*}
    admits two solutions $\tau_1 < \tau_2$ (or a degenerate solution $\tau_1 = \tau_2 = \tau_0$), and \eqref{eq:less_than_c} holds for any $\tau \in [\tau_1, \tau_2]$.
    
    Next, to maximize $N_*$ under $N = N_{\max}$ and $\tau \in [\tau_1, \tau_2]$, we choose
    \[ 
     \tau_*=\arg\max_{\tau\in [\tau_1,\tau_2]}\frac{N_{\max}-1}{\exp(\tau^{-1})P_{N_{\max}} (\tau,d,\bsgamma)}=\arg\min_{\tau\in [\tau_1,\tau_2]}\exp(\tau^{-1})P_{N_{\max}} (\tau,d,\bsgamma).
     \]
     Since
     \begin{equation}\label{eq:logderiative2}
         \tau\frac{d}{d\tau}\log\left(\exp \left(\tau^{-1}\right)P_{N_{\max}} (\tau,d,\bsgamma)\right)=-\frac{1}{\tau}+\sum_{j=1}^d \frac{2\gamma_j^{1/(2\alpha)}\tau \log N_{\max} }{1+ 2\gamma_j^{1/(2\alpha)} (1+\tau\log N_{\max})}
     \end{equation}
      is strictly increasing in $\tau>0$, it has a unique zero $\tau'_0$. Comparing \eqref{eq:logderiative1} and \eqref{eq:logderiative2}, we find $\tau'_0<\tau_0\leq \tau_2$, and hence $\tau_*=\max(\tau'_0,\tau_1)$. Moreover,
      \[ 
      \frac{1}{\tau'_0}=\sum_{j=1}^d \frac{2\gamma_j^{1/(2\alpha)}\tau'_0 \log N_{\max} }{1+ 2\gamma_j^{1/(2\alpha)} (1+\tau'_0\log N_{\max})}<d,
      \]
      so $\tau_*\geq \tau'_0>d^{-1}$.
      
      With this choice of $\tau = \tau_*$, $N = N_{\max}$, and $R$ the odd integer satisfying \eqref{eq:choose_R_coro}, our proposed algorithm solves the $L_2$-approximation problem within the budget $RN \leq M_{\max}$, and achieves an error of order
      \[ 
      N^{-\alpha}_*\sqrt{2\tau^{-1}+1 + \frac{2N}{N-1}\log(N-1)}\leq N^{-\alpha}_*\sqrt{2d+1 + \frac{2N}{N-1}\log(N-1)}
      \]
      with probability at least $1-\delta$. 
\end{remark}

\section{Numerical experiments}\label{sec:numerics}
Here, we present numerical experiments. As proven in \cite{byrenheid2017tight}, using a single rank-1 lattice rule with $M$ points cannot achieve a convergence rate better than $\mathcal{O}(M^{-\alpha/2})$ for the worst-case $L_2$-approximation error in any dimension $d\ge 2$. The main objective here is to observe the empirical performance of the proposed median lattice-based algorithm, with particular focus on whether it can attain a faster convergence rate in practice. To this end, we restrict our attention to the two-dimensional case, i.e., $d = 2$.

We test the following two periodic functions, both of which are used in \cite{CGK24}:
\begin{align*}
f_1(\bsx) &=  \prod_{j=1}^d \frac{121\sqrt{33}}{100}\max \left\{\frac{25}{121}-\left(x_j-\frac{1}{2} \right)^2, 0 \right \}, \\
f_2(\bsx) &= \prod_{j=1}^{d} \left(x_j-\frac{1}{2} \right)^2\sin(2\pi x_j-\pi),
\end{align*}
with $d=2$. The first test function, also used in \cite{kammerer2019approximation}, is a scaled periodized kink function that belongs to the space $\Hcal_{d,3/2-\epsilon,\bsgamma}$ for arbitrarily small $\epsilon > 0$. This means that the smoothness of this function is arbitrarily close to $3/2$. The second test function is smoother, belonging to the space $\Hcal_{d,5/2-\epsilon,\bsgamma}$ for arbitrarily small $\epsilon > 0$, indicating that the smoothness of this function is arbitrarily close to $5/2$.

To run \Cref{alg:median}, we set $\alpha = 3/2$ for $f_1$ and $\alpha = 5/2$ for $f_2$. For both test functions, we use the same weight parameters, $\gamma_1 = \gamma_2 = 1$. The total computational budget $M_{\max}$ is always chosen to be a power of 2. The number of points $N$ used in each lattice rule, the number of repetitions $R$, and the parameter $\tau$ are determined according to the procedure described in \Cref{rmk:choosingtau} with $\delta = 0.01$. Specifically, $N$ is set to the largest prime number satisfying \eqref{eq:maximum_prime}, and $\tau$ is set to $\tau_*$ as defined in \Cref{rmk:choosingtau}. The number of repetitions $R$ is then taken to be the largest odd integer such that $R \leq M_{\max}/N$. Since we are working in the two-dimensional unweighted setting, these parameter choices may not be optimal, but they are adopted for the sake of consistency and objectivity. We do not check throughout whether the condition \eqref{eq:less_than_c} with $c = \exp(-1)$ is satisfied.

Following \cite{CGK24}, we compute the squared $L_2$-approximation error exactly as
\begin{align*}
    &\|f-A^{\rand}_{N,\bsz_{\{1{:}R\}},\bsDelta_{\{1{:}R\}},\Acal_d(N_*)}(f)\|^2_{L_2}\\
    & = \sum_{\bsh\notin \mathcal{A}_d(N_*)} |\widehat{f}(\bsh)|^2+\sum_{\bsh\in \Acal_d(N_*)}|\widehat{f}_{N,\bsz_{\{1{:}R\}},\bsDelta_{\{1{:}R\}}}(\bsh)-\widehat{f}(\bsh)|^2\\
    & = \|f\|_{L_2}^2-\sum_{\bsh\in \mathcal{A}_d(N_*)} |\widehat{f}(\bsh)|^2+\sum_{\bsh\in \Acal_d(N_*)}|\widehat{f}_{N,\bsz_{\{1{:}R\}},\bsDelta_{\{1{:}R\}}}(\bsh)-\widehat{f}(\bsh)|^2,
\end{align*} 
where the $L_2$-norm and the Fourier coefficients of $f \in \{f_1, f_2\}$ are computed analytically in advance. With our choice $\delta = 0.01$, we expect that the error bound in \Cref{cor:error_bound} holds with probability at least $0.99$, so that we do not examine the stability of our algorithm by running it multiple times.

We now present the results of our numerical experiments. \Cref{fig:results} shows the $L_2$-approximation error of our median lattice-based algorithm plotted against $M$ on a log-log scale. In this plot, we set $M = N R$ using the actual values of $N$ and $R$, rather than using $M = M_{\max}$. The left and right panels correspond to the test functions $f_1$ and $f_2$, respectively. Each panel includes three reference lines: $M^{-\alpha/2}$, $M^{-3\alpha/4}$, and $M^{-\alpha}$. The first line represents the optimal rate possibly achievable by a lattice-based algorithm with a single rank-1 lattice point set, while the third line indicates the best possible rate for any algorithm with $M$ function evaluations. The second line lies between these two, serving as an intermediate reference.

For both $f_1$ and $f_2$, the observed convergence rate is approximately $M^{-3\alpha/4}$, which falls short of the expected, optimal rate $M^{-\alpha}$. Nevertheless, the decay is clearly faster than $M^{-\alpha/2}$. For $f_2$, we also observe non-asymptotic behavior: the error decays more slowly for small $M$, but gradually approaches the $M^{-3\alpha/4}$-rate as $M$ increases.

\begin{figure}[htbp]
    \centering
    \includegraphics[width=0.45\linewidth]{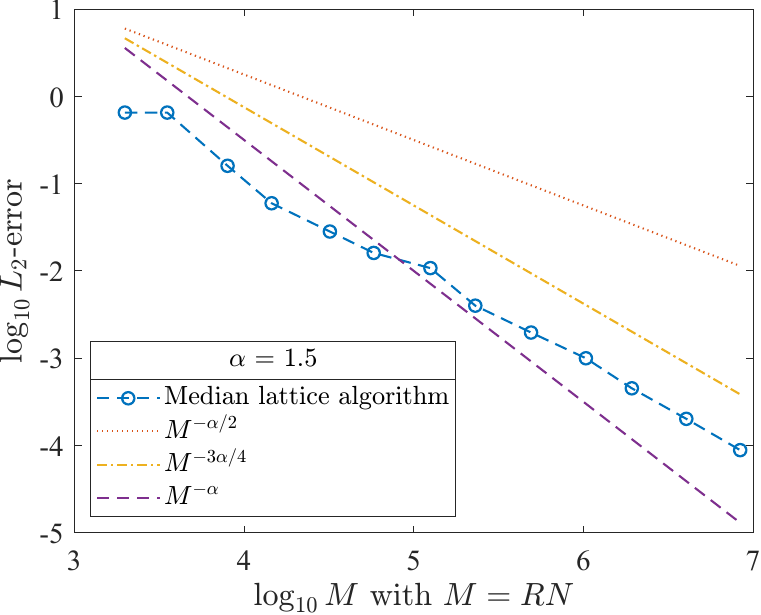}\, 
    \includegraphics[width=0.45\linewidth]{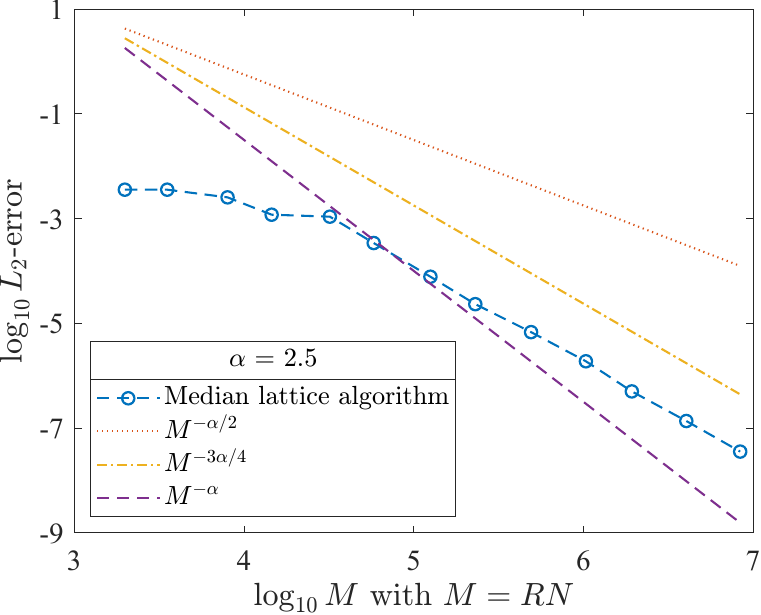}
    \caption{Results for the test functions $f_1$ (left) and $f_2$ (right): $L_2$-approximation error versus the total number of function evaluations $M$.}
    \label{fig:results}
\end{figure}

To confirm that the observed suboptimal convergence rates result from finite-sample effects, we present \Cref{fig:results2} and \Cref{fig:results3}. \Cref{fig:results2} shows the $L_2$-approximation error plotted against $N_*$, instead of $M$, on a log-log scale. \Cref{fig:results3} displays $N_*$ versus $M$, also on a log-log scale. Since we set $\gamma_1 = \gamma_2 = 1$, the value of $N_*$ is independent of $\alpha$ (see~\eqref{eqn:Nstar}), and therefore the plots for $f_1$ and $f_2$ are identical. For this reason, we present only a single plot in \Cref{fig:results3}. Note that, since \Cref{fig:results3} only illustrates how $N_*$ depends on $M$ and does not involve computing the $L_2$-approximation error, we extend the range of $M$ beyond that used in the other figures.

We observe from \Cref{fig:results2} that the convergence rate of $N_*^{-\alpha}$ is achieved for both $f_1$ and $f_2$. This agrees well with the theoretical result established in \Cref{thm:L2rate_constants}. On the other hand, \Cref{fig:results3} illustrates that the relationship between $N_*$ and $M$ is still suboptimal, which accounts for the slower convergence rate observed in \Cref{fig:results}. Ideally, as $M$ increases, $N_*$ should grow almost linearly with $M$. However, such behavior is not yet evident in the current range of $M$. For larger values of $M$, we expect $N_*$ to depend almost linearly on $M$, which in turn yields the optimal convergence rate of the $L_2$-approximation error with respect to $M$.

\begin{figure}[htbp]
    \centering
    \includegraphics[width=0.45\linewidth]{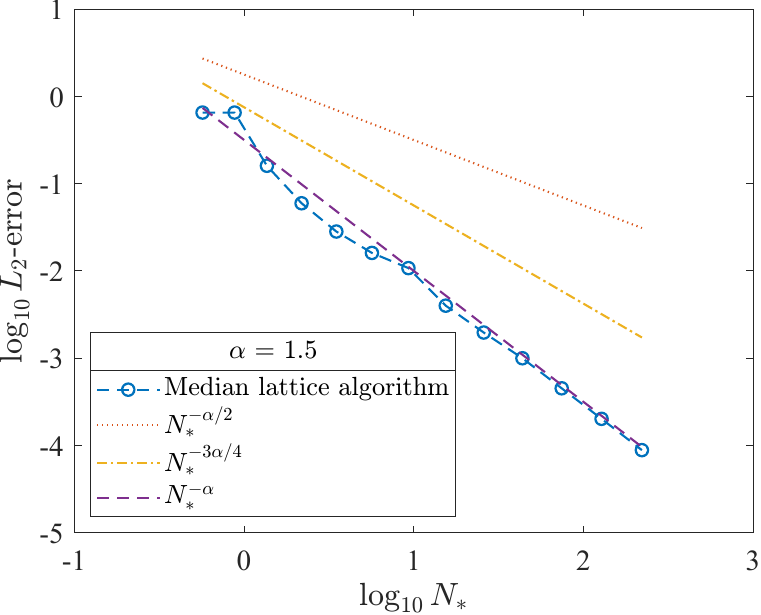}\, 
    \includegraphics[width=0.45\linewidth]{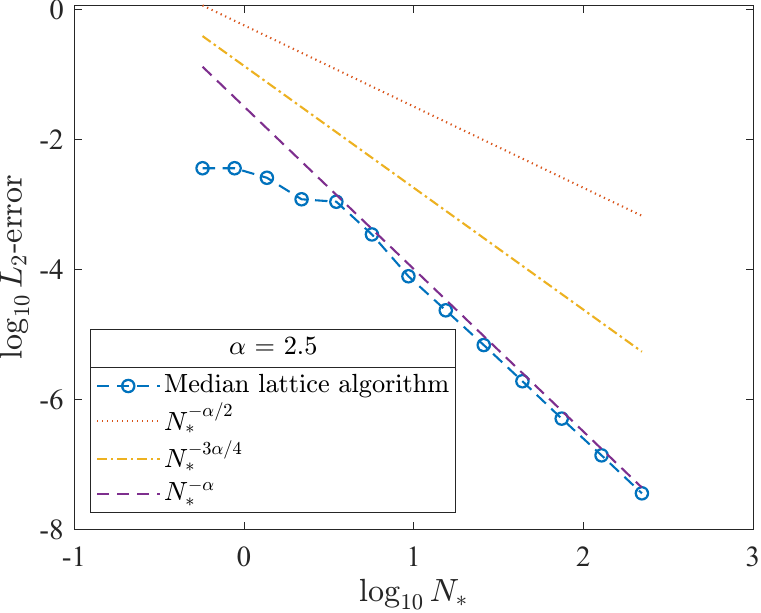}
    \caption{Results for the test functions $f_1$ (left) and $f_2$ (right): $L_2$-approximation error versus $N_*$.}
    \label{fig:results2}
\end{figure}

\begin{figure}[htbp]
    \centering
    \includegraphics[width=0.45\linewidth]{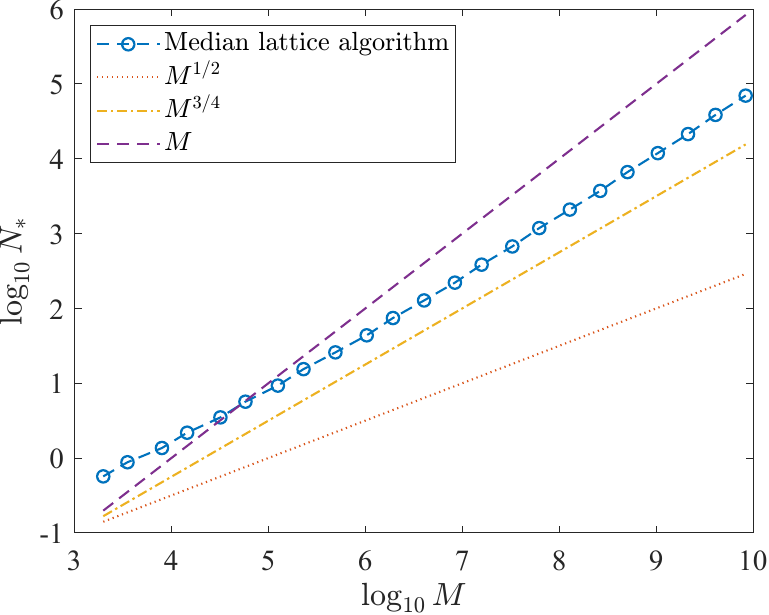}
    \caption{$N_*$ versus the total number of function evaluations $M$.}
    \label{fig:results3}
\end{figure}

\section{Discussion}\label{sec:discussion}

Thus far, we have analyzed $\mathrm{err}(f,L_2, A^{\rand}_{N,\bsz_{\{1{:}R\}},\bsDelta_{\{1{:}R\}},\Acal_d(N_*)})$, the error for a fixed $f\in\Hcal_{d,\alpha,\bsgamma}$, with a particular emphasis on the $(\epsilon,\delta)$-approximation framework \eqref{eq:eps-delta_framework}. Here, we first discuss how this result relates to the randomized error \eqref{eq:worst-case_randomized}.

Our choice of $\tau,N,R$ from \Cref{rmk:choosingtau} in fact implies a stronger bound on the failure probability. From  \eqref{eq:choose_R_first} and \eqref{eq:choose_R_coro}, the failure probability in \Cref{thm:L2rate_constants} is bounded by
      \begin{align}\label{eqn:failureasym}
           \Big(1+\frac{N-1}{1+\tau\log N_*}\Big)
    \left(\frac{4(1+\tau)}{1+\tau\log N_* }\right)^{\lceil R/2 \rceil} \leq  \delta \left(\frac{4e(1+\tau)}{1+\tau\log N_* }\right)^{\lceil R/2 \rceil}.
      \end{align}
Because $\tau=\tau_*$ attains the maximum $N_*$ over $\tau\in [\tau_1,\tau_2]$ while $\tau_0\in [\tau_1,\tau_2]$, 
     \begin{align*}
         \frac{4e(1+\tau)}{1+\tau\log N_* }\leq &\frac{4e(1+\tau)}{1+\tau\log(N-1)-\tau\log(P_N(\tau_0,d,\bsgamma))-\tau\tau^{-1}_0 } \\
         \leq &\frac{4e({\tau'_0}^{-1}+1)}{\log(N-1)-\log(P_N(\tau_0,d,\bsgamma))},
     \end{align*}
     where we have used $
    \tau=\tau_*=\max(\tau'_0,\tau_1)\leq \tau_0$. As $N=N_{\max}$ increases to infinity, \eqref{eq:logderiative1} shows that $\tau_0$ decreases to $4e/d$, and \eqref{eq:logderiative2} shows that $\tau'_0$ decreases to $1/d$. Because $R$ is proportional to $\log N$, the upper bound \eqref{eqn:failureasym} asymptotically converges to $0$ faster than any fixed polynomial rate in $M=NR$.

We can use the improved bound to show that our algorithm achieves a nearly-optimal convergence rate in terms of the randomized $L_2$-approximation error, while keeping good tractability properties as shown in the previous section. For any $f\in \Hcal_{d,\alpha,\bsgamma}$, 
we see that
\[ \|f\|_{L_2}^2 \leq \norm{f}_{d,\alpha,\bsgamma}^2 \sup_{\bsh\in \ZZ^d}\frac{1}{r_{2\alpha,\bsgamma} (\bsh)}=\norm{f}_{d,\alpha,\bsgamma}^2. \]
Moreover, by following \cite[Proof of Lemma~4.2]{CGK24}, it holds, for all $\bsh\in \Acal_d(N_*)$ and any $\bsz\in \{1{:}(N-1)\}^d$ and $\bsDelta\in [0,1)^d$, that
\begin{align*}
    & \left|\widehat{f}_{N,\bsz,\bsDelta}(\bsh)\right|^2 = \left|\sum_{\substack{\bsl\in \ZZ^d\\ \bsl-\bsh\in P_{N,\bsz}^\perp}}\widehat{f}(\bsl)\exp(2\pi\icomp (\bsl-\bsh)\cdot\bsDelta)\right|^2\\
    & \qquad \leq \sum_{\substack{\bsl\in \ZZ^d\\ \bsl-\bsh\in P_{N,\bsz}^\perp}}\left|\widehat{f}(\bsl)\exp(2\pi\icomp (\bsl-\bsh)\cdot\bsDelta)\right|^2 r_{2\alpha,\bsgamma} (\bsl) \sum_{\substack{\bsl'\in \ZZ^d\\ \bsl'-\bsh\in P_{N,\bsz}^\perp}}\frac{1}{r_{2\alpha,\bsgamma} (\bsl')}\\
    & \qquad \leq \norm{f}_{d,\alpha,\bsgamma}^2\sum_{\bsl'\in \ZZ^d}\frac{1}{r_{2\alpha,\bsgamma} (\bsl')} = \norm{f}_{d,\alpha,\bsgamma}^2\prod_{j=1}^{d}\left( 1+2\gamma_j\zeta(2\alpha)\right),
\end{align*}
where the first inequality follows from the Cauchy--Schwarz inequality.

Thus, for \emph{any} realization of $A^{\rand}_{N,\bsz_{\{1{:}R\}},\bsDelta_{\{1{:}R\}},\Acal_d(N_*)}$, the $L_2$-approximation error is always bounded by
    \begin{align*}
        & \|f-A^{\rand}_{N,\bsz_{\{1{:}R\}},\bsDelta_{\{1{:}R\}},\Acal_d(N_*)}(f)\|_{L_2} \\
        &\qquad \leq \|f\|_{L_2}+\left(\sum_{\bsh\in \Acal_d(N_*)}\left|\widehat{f}_{N,\bsz_{\{1{:}R\}},\bsDelta_{\{1{:}R\}}}(\bsh)\right|^2\right)^{1/2}\\
        &\qquad  \leq \norm{f}_{d,\alpha,\bsgamma}\left(1+\left(|\Acal_d(N_*)|\,  \prod_{j=1}^{d}\left( 1+2\gamma_j\zeta(2\alpha)\right)\right)^{1/2}\right),
    \end{align*}
where, according to \Cref{cor:AdNstarcount}, the size of $\Acal_d(N_*)$ is bounded by $N$.
This means that, even for the worst realization of $A^{\rand}_{N,\bsz_{\{1{:}R\}},\bsDelta_{\{1{:}R\}},\Acal_d(N_*)}$, the corresponding $L_2$-approximation error grows at most like $N^{1/2}$. 
However, the failure probability converges to $0$ faster than any fixed polynomial rate. In total, we have
    \begin{equation*}
        \sup_{\substack{f\in \Hcal_{d,\alpha,\bsgamma}\\ 
  \norm{f}_{d,\alpha,\bsgamma}\le 1}}
  \left(\e\|f-A^{\rand}_{N,\bsz_{\{1{:}R\}},\bsDelta_{\{1{:}R\}},\Acal_d(N_*)}(f)\|^2_{L_2}\right)^{1/2}=O(M^{-\alpha+\varepsilon}) 
    \end{equation*}
    for any $\varepsilon>0$, where the implied factor depends only polynomially on the dimension $d$, or is even independent of $d$, under the same summability conditions on the weights $\gamma_j$ we showed in \Cref{sec:tractability}.

According to the literature, this order of convergence is known to be nearly optimal for the following reasons.
In \cite{novak1992opitmal}, Novak proved that, in general separable Hilbert spaces, the optimal convergence rate of the randomized error coincides with that of the deterministic worst-case error, when algorithms are restricted to using (either random or deterministic) linear functionals as information.
Later, in \cite{krieg2021function}, Krieg and Ullrich showed that, for the Korobov space, the optimal rate of the $L_2$-approximation error---namely, $\mathcal{O}(M^{-\alpha})$---remains the same, up to logarithmic factors, whether algorithms use at most $M$ linear functionals or $M$ function evaluations.
Combining these two results, we conclude that the rate $\mathcal{O}(M^{-\alpha})$ is indeed best possible for the randomized setting with function evaluations.

Next, we discuss the worst-case error \eqref{eq:worst-case}.
Our main result, stated in \Cref{cor:error_bound}, shows that the median algorithm achieves an error bound of $\mathcal{O}(M^{-\alpha+\varepsilon})$ with high probability for individual functions. One might therefore expect that the same rate also holds for the worst-case error with high probability. However, this is not necessarily the case.
This situation is different from what was studied for the integration problem in \cite{goda2022free}, where a high-probability worst-case bound was established.
The key obstacle lies in the fact that our result guarantees the existence of a set of good randomized lattice algorithms, but this set may depend on the specific function being approximated. 
To obtain a worst-case error bound, a single set of randomized lattice algorithms must perform well for all functions in the unit ball of $\Hcal_{d,\alpha,\bsgamma}$. Consequently, extending our result to a nearly optimal rate for the worst-case error is nontrivial; straightforward arguments only yield a deteriorated rate of convergence.
Whether or not the proposed algorithm can achieve a nearly optimal rate for the worst-case error with high probability remains an open question for future work.

A further limitation of our algorithm lies in its explicit dependence on the parameters $\alpha$ and $\gamma_j$ to determine $N_*$, as mentioned in \Cref{rem:limitation}. This dependency introduces a practical challenge: the convergence rate is heavily influenced by the choice of $\alpha$ and $\gamma_j$. For instance, assuming a slower decay rate for $\gamma_j$ than necessary leads to an overly conservative $N_*$, which in turn degrades the convergence rate. A more desirable solution would be a universal algorithm that simultaneously achieves near-optimal convergence rates of order $\mathcal{O}(M^{-\alpha})$ and dimension-independent convergence under suitable decay conditions on $\gamma_j$, all without requiring prior knowledge of $\alpha$ and $\gamma_j$. Designing such an algorithm, however, poses a significant challenge, as it would need to adapt implicitly to the intrinsic smoothness and weight structure of the target function space. We leave this critical open problem for future investigation.

\section*{Acknowledgements}
The first and second authors acknowledge the support of the Austrian Science Fund (FWF) Project  P 34808/Grant DOI: 10.55776/P34808. The third author acknowledges the support of JSPS KAKENHI Grant Number 23K03210. For open access purposes, the authors have applied a CC BY public copyright license to any author accepted manuscript version arising from this submission. Moreover, the authors would like to thank F.~Bartel, D.~Krieg, and I.H.~Sloan for valuable comments. The second author would like to thank F.Y.~Kuo and I.H.~Sloan for their hospitality during his visit to UNSW Sydney, where parts of this paper were written.


\end{document}